\documentclass[11pt]{amsart}

\usepackage[dvipsnames]{xcolor}
\usepackage{amssymb,verbatim}
\usepackage{amsmath,amsfonts,enumitem}
\usepackage[mathscr]{euscript} 
\usepackage{amsthm}
\usepackage{url}
\usepackage{graphicx} 
\usepackage{enumitem} 
\usepackage{hyperref} 
\hypersetup{colorlinks} 
\usepackage{bbm} 
\usepackage{bm} 
\usepackage{mathrsfs} 
\usepackage{cancel} 
\usepackage{ytableau} 

\usepackage[colorinlistoftodos]{todonotes}

\RequirePackage{cleveref}
\usepackage{hypcap}
\hypersetup{colorlinks=true, citecolor=darkblue, linkcolor=darkblue}
\definecolor{darkblue}{rgb}{0.0,0,0.7}
\newcommand{\darkblue}{\color{darkblue}}

\definecolor{darkred}{rgb}{0.68,0,0}
\newcommand{\darkred}{\color{darkred}}

\definecolor{darkgreen}{rgb}{0,.38,0}
\newcommand{\darkgreen}{\color{darkgreen}}

\newcommand{\defn}[1]{\emph{\darkblue #1}}
\newcommand{\defna}[1]{\emph{\darkred #1}}

\newcommand{\defng}[1]{\emph{\darkgreen #1}}

\setlist[enumerate]{
	label=\textnormal{({\roman*})},
	ref={\roman*}}

\makeatletter
\def\th@plain{%
	\thm@notefont{}
	\itshape 
}
\def\th@definition{%
	\thm@notefont{}
	\normalfont 
}
\makeatother

\newtheorem{thm}{Theorem}[section]

\newtheorem*{claim*}{Claim}
\newtheorem{cor}[thm]{Corollary}

\newtheorem{prop}[thm]{Proposition}

\newtheorem{conv}[thm]{Convention}

\theoremstyle{definition}
\newtheorem{ex}[thm]{Example}
\newtheorem{rem}[thm]{Remark}

\numberwithin{figure}{section}
\numberwithin{equation}{section}

\def\zz{\mathbb Z}
\def\nn{\mathbb N}
\def\cc{\mathbb C}

\def\rr{\mathbb R}
\def\qqq{\mathbb Q}

\def\pp{\mathbb P}

\def\sm{\smallsetminus}

\def\la{\lambda}
\def\ga{\gamma}
\def\si{\sigma}
\def\de{\delta}

\def\al{\alpha}
\def\be{\beta}
\def\om{\omega}

\def\vk{\varkappa}

\def\cC{\mathcal C}

\def\cD{\mathcal D}

\def\cP{\mathcal P}
\def\cq{\mathcal Q}
\def\cQ{\mathcal Q}

\def\cZ{\mathcal Z}

\def\ssu{\subset}

\def\wt{\widetilde}
\def\<{\langle}
\def\>{\rangle}

\def\ComW{\text{\small {\rm W}}}
\def\vrV{\text{\small {\rm V}}}

\def\height{\text{{\rm height}}}

\def\GL{ {\text {\rm GL} } }

\def\rR{ {\textsc {\rm R} } }

\def\rU{ {\text {\rm U} } }

\def\Ups{{\small {\Upsilon}}}

\def\0{{\mathbf 0}}

\def\.{\hskip.06cm}
\def\ts{\hskip.03cm}

\def\bx{{\textbf{\textit{x}}}}

\def\ba{\textbf{\textbf{a}}}

\def\bc{\textbf{\textbf{c}}}

\def\La{\Lambda}

\newcommand{\ord}{\mathrm{ord}}

\newcommand{\SSYT}{\operatorname{SSYT}}

\def\atom{\textsf{{atom}}}
\def\atomrm{\textsf{\em {atom}}}

\def\.{\hskip.06cm}
\def\ts{\hskip.03cm}

\def\nin{\noindent}

\newcommand{\textsu}[1]{\textup{\textsf{#1}}}

\newcommand{\ComCla}[1]{\textup{\textsu{#1}}}

\newcommand{\sharpP}{\ComCla{\#P}}
\newcommand{\SP}{\ComCla{\#P}}
\newcommand{\SBQP}{\ComCla{\#BQP}}

\newcommand{\GapP}{\ComCla{GapP}}
\newcommand{\GapFP}{\ComCla{GapP/FP}}
\newcommand{\GapPP}{\ComCla{GapP}_{\ge 0}}

\newcommand{\Sigmap}{\ensuremath{\Sigma^{{\textup{p}}}}}

\newcommand{\NP}{\ComCla{NP}}

\renewcommand{\P}{\ComCla{P}}

\newcommand{\PH}{\ComCla{PH}}

\newcommand{\FP}{\ComCla{FP}}

\def\SP{\sharpP}
\def\poly{{\P}}

\newcommand{\inv}{\operatorname{{\rm inv}}}

\def\QSym{ {\text {\sc QSym} } }

\newcommand{\Sc}{\mathfrak{S}}
\newcommand{\FS}{\mathfrak{F}}
\newcommand{\code}{{\sf{code}}}
\newcommand\wtx{{\tt wt}}

\title[Signed combinatorial interpretations]{Signed combinatorial interpretations \\ in algebraic combinatorics}

\author[Igor Pak \. \and \. Colleen Robichaux]{Igor Pak$^\star$  \. \and \.  Colleen Robichaux$^\star$}

\thanks{\thinspace ${\hspace{-.45ex}}^\star$Department of Mathematics,
UCLA, Los Angeles, CA 90095, USA. Email:  \texttt{\{pak,robichaux\}@math.ucla.edu}}

\thanks{\today}

\begin{document}

\begin{abstract}
We prove the existence of signed combinatorial interpretations for
several large families of structure constants.  These families include
standard bases of symmetric and quasisymmetric polynomials, as well as
various bases in  Schubert theory.  The results are stated in the
language of computational complexity, while the proofs are based on the
effective M\"obius inversion.
\end{abstract}

\maketitle

\nin

\section{Introduction}\label{s:intro}

\subsection{Foreword} \label{ss:intro-foreword}
Back in 1969, the renown British philosophers Jagger and Richards made a profound if
not wholly original observation, that ``you can't always get what you want'' \cite{JR69}.
Only recently, this piece of conventional wisdom started to gain acceptance in algebraic
combinatorics, at least when it comes to \defna{combinatorial interpretations}, cf.\
\cite{Pak}.

The interest in combinatorial interpretations is foundational in the area,
and goes back to the early papers by Young, where he defined what we now call
{standard Young tableaux}, to count degrees of irreducible $S_n$ characters.
After over a century of progress and numerous successes finding combinatorial
interpretations, a dozen or so open problems remain a thorn in our side
(see e.g.\ \cite{Sta00} and \cite{Pak-OPAC}).

Negative results started to arrive a few years ago, all of the form
``a combinatorics interpretation of [some numbers] implies a collapse
of the polynomial hierarchy'', see \cite{CP23,IP22,IPP22}.   Informally,
these conditional results show that a major conjecture in theoretical
computer science, a close relative to \ts $\poly \ne \NP$, implies that
combinatorial interpretations \emph{do not exist} in some cases.

Setting aside the precise technical meaning of a ``combinatorial
interpretation'' (see below), one can ask what's the best one can do?
What should be the right replacement of a ``combinatorial interpretation''?
We suggest that \defna{signed combinatorial interpretations} \ts
are a perfect answer to these questions.

To understand why, first note that having signs is unavoidable when the numbers
can be both positive and negative, and when the sign is hard to compute.
A prototypical example is the $S_n$ character value \ts $\chi^\la(\mu)$ \ts
given by the Murnaghan--Nakayama rule (see below).  This is why it's so
puzzling to see signed combinatorial interpretations for nonnegative structure
constants when no (unsigned) combinatorial interpretation is known, e.g.\
the Kronecker and Schubert coefficients.

In this paper, we make a systematic study and prove signed combinatorial
interpretations for structure constants
for many families of symmetric functions, their relatives, and generalizations.
We present signed combinatorial interpretations in all cases without fail,
leading to the following meta observation:

\smallskip

\begin{center}\begin{minipage}{7.5cm}%
{\hskip-.5cm\defn{In algebraic combinatorics, all integral constants have signed combinatorial
interpretations.}}
\end{minipage}\end{center}

\smallskip

In conclusion, let us mention that signed combinatorial interpretations
are useful for many applications, ranging from upper bounds to asymptotics,
from analysis of algorithms to computational complexity, see below.
As it turns out, ``if you try sometime you'll find you get what you need'',
fully consistent with the philosophy in~\cite{JR69}.

\smallskip
\subsection{Signed combinatorial interpretations} \label{ss:intro-signed}
We postpone the discussion of computational complexity classes until
later in this section.  Let us briefly remind the reader of the
complexity classes \ts $\NP$, \ts $\SP$, \ts and \ts $\GapP$.

Denote \ts $W:=\{0,1\}^\ast$ \ts and \ts $W_n :=\{0,1\}^n$ \ts be sets of \emph{words}.
The length \ts $|w|$ \ts is called the \emph{size} of~$w\in W$, so \ts $|w|=n$ \ts
for all \ts $w\in W_n\ts$.
A \emph{language} \ts is defined as \ts $A \subseteq W$.
Denote \ts $A_n :=  A \cap W_n$.  We say that \ts $A$ \ts
is in \ts $\NP$ \ts if the membership \ts $[w\in^? A]$ \ts can
be decided in polynomial time (in the size $|w|$),
by a Turing Machine.  We think of \ts $A_n$ \ts
as the set of \emph{combinatorial objects} of size~$n$.  For example, partitions,
Young tableaux, permutations, etc., are all combinatorial objects when encoded
appropriately.

Let \ts $f: W\to \nn$ \ts be an integer function.
We say that $f$ has a \defn{combinatorial interpretation}
(also called a \emph{combinatorial formula}), if \ts $f\in \SP$.
This means that there is a language \ts $B\subseteq W^2$ \ts
in~$\NP$, such that for all \ts $w\in W$ \ts we have \ts
$f(w) = \#\{u \. : \. (w,u) \in B\}$.  Informally, we say that the language
\ts $B$ \ts \defn{counts the numbers}~$\{f(w)\}$.
Note that the decision problem \ts $[f(w)>^?0]$ \ts is also
in~$\NP$. We refer to \cite{Pak-OPAC} for the background
and many examples.

In algebraic combinatorics, all standard combinatorial interpretations are
in~$\SP$.  These include \emph{character degrees} \ts $f^\la = \chi^\la(1)$, \ts
\emph{Kostka numbers} \ts $K_{\la\mu}$, \ts and the \emph{Littlewood--Richardson coefficients}
\ts $c^\la_{\mu\nu}$, see e.g.\ \cite{Mac95,Sag01,Sta-EC}.
Indeed, in each case these are counted by certain Young tableaux,
the validity of which can be verified in polynomial time.\footnote{There is a subtlety
here in the encoding of partitions.  Here and throughout the paper we use
\emph{unary encoding}.  Although some of our results generalize to the binary encoding,
this would require more involved arguments, see~$\S$\ref{ss:finrem-binary}.}

Let \ts $f: \{0,1\}^\ast\to \zz$ \ts be an integer function.
We say that function $f$ has a \defn{signed combinatorial interpretation}
(also called a \emph{signed combinatorial formula}),
if \ts $f = g-h$ \ts for some \ts $g,h\in \SP$.  The set of such functions
is denoted \ts $\GapP := \SP-\SP$.  This class is well studied in the
computational complexity literature, see e.g.\ \cite{For97,HO02}.
We use \ts $\GapPP$ \ts to denote the subset of nonnegative \ts $f \in \GapP$.
Note that \ts $[f(w)>^?0]$ \ts is not always in~$\NP$ (see below).

In algebraic combinatorics, there are many natural examples of signed combinatorial
interpretations. As we mentioned above, these include \emph{character values} \ts
$\chi^\la(\mu)$ \ts via the \emph{Murnaghan--Nakayama rule}
(see e.g.\ \cite{JK81,Sag01,Sta-EC}).  Another example is the
\emph{inverse Kostka numbers} \ts $K^{-1}_{\la\mu}$ \ts defined as the entry
in the inverse Kostka matrix \ts $(K_{\la\mu})^{-1}$, given by the
\emph{E\u{g}ecio\u{g}lu--Remmel rule} \cite{ER90}.  In both cases,
the rules subtract the number of certain rim hook tableaux, some with
a positive sign and some with a negative, where the sign is easily computable.

As we mentioned above, there are many examples of signed combinatorial interpretations
for nonnegative functions, some of which are collected in \cite{Pak-OPAC}.
Famously, \emph{Kronecker coefficients} \ts $g(\la,\mu,\nu)$ \ts are given by the
large signed summation of the numbers of $3$-dimensional contingency arrays,
see \cite{BI08,CDW12,PP17}.  Another celebrated example is the \emph{Schubert coefficients} \ts
(also called \emph{Schubert structure constants}) \ts $c^\ga_{\al\be}$ \ts
given by the \emph{Postnikov--Stanley formula} \ts  in terms of the number of
chains in the Bruhat order \cite[Cor.~17.13]{PS09}.

\smallskip
\subsection{Main results} \label{ss:intro-main}
Let \ts $\rR$ \ts be a ring and let \ts $\Ups:=\{\xi_\al\}$ \ts be a linear basis
in~$\rR$, where the indices form a set \ts $A=\{\al\}$ \ts of combinatorial
objects.
The \defn{structure constants} \ts $\big\{c(\al,\be,\ga)\big\}$ \ts
for \ts $\Ups$ \ts are defined as
$$
\xi_\al \cdot \xi_\be \, = \, \sum_{\ga\in A} \. c(\al,\be,\ga) \. \xi_\ga \ \ \text{where} \ \ \al,\be\in A\ts.
$$
When the structure constants are integral, one can ask whether the function \ts $\bc: A^3 \to \zz$ \ts
is in $\GapP$, i.e.\
has a signed combinatorial interpretation.   Additionally, when they are nonnegative,
one can ask if \ts $\bc$ \ts is in $\SP$, i.e.\ has (the usual) combinatorial interpretation.

We postpone the definitions of various rings and bases to later sections, moving
straight to results both known and new.  The following result is routine,
well-known, and is included both for contrast and for completeness:

\begin{thm}[{\rm classic structure constants}{}]\label{t:main-classic}
Let \ts $\La_n = \cc[x_1,\ldots,x_n]^{S_n}$ \ts denote symmetric polynomials in $n$ variables.
The following bases in \ts $\La_n$ \ts have structure constants in~$\SP:$
    \begin{itemize}
       \item Schur polynomials \ts $\{s_\la\.:\. \ell(\la) \le n\}$,
       \item monomial symmetric polynomials \ts $\{m_\la\.:\. \ell(\la) \le n\}$,
       \item power sum symmetric polynomials \ts $\{p_\la\.:\. \la_1 \le n\}$,
       \item elementary symmetric polynomials \ts $\{e_\la\.:\. \la_1 \le n\}$, and
       \item complete homogeneous symmetric polynomials \ts $\{h_\la\.:\. \la_1 \le n\}$.
\end{itemize}
\end{thm}

The last four of these items are completely straightforward and follow
directly from the definition.   On the other hand, the first item corresponding
to the Littlewood--Richardson (LR) coefficients is highly nontrivial.  Rather than
reprove it, we give an elementary proof of a weaker claim, that
LR-coefficients are in~$\GapP$.   This sets us up for several generalizations.

\smallskip

First, we consider deformations of Schur polynomials.  Let \ts $q,t,\al \in \qqq$ \ts
be fixed rational numbers.

\begin{thm}\label{t:main-qt}
    The following bases in $\La_n$ have structure coefficients in~$\GapFP:$
    \begin{itemize}
        \item Jack symmetric polynomials \ts $\{P_{\la}(x;\alpha)\.:\. \ell(\la) \le n\}$, where \ts $\al>0$,
        \item Hall--Littlewood polynomials \ts $\{P_{\la}(x;t)\.:\. \ell(\la) \le n\}$, where \ts $0\le t <1$,  and
        \item Macdonald symmetric polynomials \ts $\{P_{\la}(x;q,t)\.:\. \ell(\la) \le n\}$, where \ts $0\le q,t <1$.
    \end{itemize}
\end{thm}

Here \ts $\GapFP$ \ts is a class of rational functions which can be written as \ts
$f/g$ \ts where $f\in \GapP$ \ts and \ts $g\in \FP$ \ts is a function which can be computed in
polynomial time.  Essentially, we prove that these structure constants are rational, the
numerators have signed combinatorial interpretations, and the denominators have a nice
product formula.

\smallskip

Second, we consider quasisymmetric polynomials which are somewhat
intermediate between symmetric and general polynomials:

\begin{thm}[{\rm quasisymmetric structure constants}{}]\label{t:main-quasi}
Let \ts $\QSym_n\subseteq\cc[x_1,\ldots,x_n]$ \ts be the ring of quasisymmetric polynomials
in $n$ variables.  The following bases have structure constants in~$\SP:$
    \begin{itemize}
        \item monomial quasisymmetric polynomial \ts $\{M_{\al}\}$, and
        \item fundamental quasisymmetric polynomials \ts $\{F_{\al}\}$.
\end{itemize}
The following bases have structure constants in~$\GapP:$
    \begin{itemize}
        \item dual immaculate polynomials \ts $\{\Sc_\al^*\}$, and
        \item quasisymmetric Schur polynomials \ts $\{\mathcal{S}_\al\}$.
\end{itemize}
The following  bases have structure constants in~$\SP/\FP:$
    \begin{itemize}
        \item  type 1 quasisymmetric power sum \ts
        $\{\Psi_\al\}$,
        \item type 2 quasisymmetric power sum \ts
        $\{\Phi_\al\}$, and
        \item combinatorial quasisymmetric power sum \ts
        $\{\mathfrak{p}_\al\}$.
\end{itemize}
Here we have \ts $\al$ \ts is a composition into at most $n$ positive parts.
\end{thm}

\smallskip

The first two items go back to Gessel \cite{Ges84}, while the rest are new.
Next, recall that Schubert polynomials mentioned above generalize Schur polynomials
and form a linear basis in the ring of \emph{all} \ts polynomials.  The following result
can be viewed as the analogue of Theorem~\ref{t:main-classic} in this more general setting.

\begin{thm}[{\rm polynomial structure constants}{}]\label{t:main-schubert}
In the ring of polynomials $\cc[x_1,\ldots,x_n]$, the following bases have structure
constants in~$\SP:$
\begin{itemize}
         \item monomial slide polynomials \ts $\{\mathfrak{M}_{\al}\}$, and
         \item fundamental slide polynomials \ts $\{\FS_\al\}$.
\end{itemize}
The following bases have structure
constants in~$\GapP:$
    \begin{itemize}
         \item Demazure atoms \ts $\{{\atomrm}_{\al}\}$,
         \item key polynomials \ts $\{\kappa_\al\}$,
         \item Schubert polynomials \ts $\{\mathfrak{S}_{\al}\}$,
         \item Lascoux polynomials \ts $\{\mathfrak{L}_{\al}\}$, and
    \item Grothendieck polynomials \ts $\{\mathfrak{G}_{\al}\}$.
    \end{itemize}
Here we have \ts $\al\in \nn^n$ \ts is a composition into $n$ nonnegative parts.
\end{thm}

As we mentioned above, the result for Schubert polynomials follows from
the work of Postnikov and Stanley \cite{PS09}.  We reprove this result
in a simpler (but closely related) way, leading to results
in other cases.

Finally, we consider \defn{plethysm}.  This is a general notion (see \cite{NR85}), which is
especially natural in the ring of symmetric functions, where it can be defined as follows.
Let \ts $\pi: \GL(V)\to \GL(W)$ \ts and $\rho: \GL(W)\to \GL(U)$ \ts be polynomial
representations of the general linear group.  One can define $\rho[\pi]:=\rho\circ \pi$
to be the {composition} of these representations (cf.\ \cite[$\S$I.8,App.~I.A]{Mac95}
and \cite[$\S$5.4]{JK81}).
At the level of characters, the composition above corresponds to \emph{plethysm}
of symmetric functions, and gives \defn{plethysm coefficients}:
$$
f\ts [\ts g\ts ] \, = \, \sum_{\la} \. a^\la_{\mu\nu} \. s_\la \quad \text{where} \ \ \ f,g\in \La\ts,
$$
and \ts $\La$ \ts is the inverse limit of \ts $\La_n$ \ts in the category of
graded rings (see e.g.\ \cite[$\S$I.2]{Mac95}).
It was shown by Fischer and Ikenmeyer \cite[$\S$9]{FI20} that plethysm coefficients
for \ts $s_{\la}[s_\mu]$ \ts are \ts $\GapP$-complete, so in particular they are
in \ts $\GapP$, see also~$\S$\ref{ss:finrem-hist}.  The following result is a
generalization to other bases listed in Theorem~\ref{t:main-classic}.

\begin{thm}[{\rm plethysm coefficients}{}]\label{t:main-plethysm}
Let \ts $\{f_\la\}$ \ts and \ts $\{g_\la\}$ \ts be families of
symmetric polynomials from the following list of linear bases$:$
$$
\{s_\la\}, \ \ \{m_\la\}, \ \ \{p_\la\}, \ \ \{e_\la\}, \ \ \{h_\la\}\ts.
$$
Then the corresponding plethysm coefficients \ts $\{a^\la_{\mu\nu}\}$ \ts are in $\GapP$.
\end{thm}

\smallskip
\subsection{Background and motivation} \label{ss:intro-back}
Structure constants are fundamental in algebraic combinatorics, reflecting
both the advances and the challenges posed by the nature of symmetric objects.
They are a succinct trove of information reducing algebraic structures to
explicit combinatorial objects.

Having a combinatorial interpretation of numbers does more than put a
face to the name.  It reveals a combinatorial structure which in turn is
a shadow of a rich but non-quantitative geometric or algebraic structure.
Different combinatorial interpretations give different shadows, helping
to understand the big picture.

For example, on a technical level, standard Young tableaux are the leading
terms in a natural linear basis of irreducible $S_n$ modules, a simple
bookkeeping tool for a large data structure (cf.\ \cite{JK81,Sag01}).
But on a deeper level, they are a byproduct of the \emph{branching rule},
which in turn comes from $S_n$ having a long subgroup chain (cf.~\cite{OV96}).

It would be impossible to overstate the impact of combinatorial interpretations
in algebraic combinatorics and symmetric function theory, as they completely
permeate the area (see e.g.\ \cite{Mac95,Sta-EC}).  The
LR-coefficients alone have over 15 different combinatorial
interpretations (see \cite[$\S$11.4]{Pak-OPAC}), and are the subject
of hundreds of papers.

Additionally, having a combinatorial interpretation does
wonders for applications of all kinds.  For example, for the LR-coefficients,
these include the \emph{saturation theorem} \cite{KT99},  efficient algorithm for
positivity \ts $[c^\la_{\mu\nu}>^?0]$, see \cite{DM06,BI13}, and various lower and
upper bounds, see \cite{PPY19}.  Even when combinatorial interpretations
are known only in special cases, the remarkable applications follow.
For example, for the Kronecker coefficients, these include unimodality
\cite{PP13}, and $\NP$-hardness of the positivity $[g(\la,\mu,\nu)>^?0]$ proved
in \cite{IMW17}.

Naturally, a signed combinatorial interpretation is inherently less powerful
than the usual (unsigned) one.  And yet, this is usually the best known tool
to obtain any results at all.  For example, for the Kronecker coefficients,
the signed combinatorial interpretation mentioned above
gives both a fast algorithm to compute the numbers, see \cite{PP17}, and
a sharp upper bound in some cases, see~\cite{PP20}.  Of course,
lower bounds cannot be obtained this way, which is why finding a good lower
bound for the Kronecker coefficients remains a major open problem,
see e.g.\ \cite{Pan23}.

An interesting case study is the Murnaghan--Nakayama (MN) rule for the $S_n$
characters values $\chi^\la(\mu)$,  defined as the signed sum over certain
rim hook (ribbon) tableaux (see e.g.\ \cite{JK81,Sag01}).  The rule was used in
\cite{LS08,Roi96} to obtain upper bounds for character values, which in turn
give upper bounds for the mixing times of random walks on $S_n$ generated
by conjugacy classes.

In a surprising development, when the conjugacy class $\mu$ is a rectangle,
it is known that all rim hook tableaux given by the MN rule have the same sign,
see \cite{SW85}.
This observation led to rich combinatorial developments, including combinatorial
proofs of character orthogonality \cite{Whi83,Whi85}, applications
in probability \cite{Bor99}, tilings \cite{Pak-ribbon}, and the LLT polynomials
describing representations of Hecke algebra at roots of unity \cite{LLT97}
(cf.\ Remark~\ref{rem:LLT}).

In a different direction, a signed combinatorial interpretation coming from
the \emph{Frobenius formula} (that is somewhat different from the MN rule),
was used to show that deciding positivity \ts $[\chi^\la(\mu)>0]$ \ts
of the character value is $\PH$-hard \cite[Thm~1.1.5]{IPP22}.  Moreover, the authors
show that character absolute value has no combinatorial interpretation
\cite[Thm~1.1.3]{IPP22}.  More precisely, \ts $|\chi^\la(\mu)|$ \ts is not in~$\SP$
unless \ts $\Sigmap_2=\PH$.  In other words, one should not expect the (usual)
combinatorial interpretation for the character values, unless one believes that
the polynomial hierarchy collapses.

In a warning to the reader, we should emphasize that some natural
combinatorial formulas defining the
numbers above are not, in fact, \ts $\GapP$ \ts formulas.  For example, the
definition of Kronecker coefficients gives:
$$
g(\la,\mu,\nu) \. := \. \<\chi^\la,\chi^\mu \cdot \chi^\nu\> \. = \.
\frac{1}{n!} \.\sum_{\si\in S_n} \chi^\la(\si) \. \chi^\mu(\si)\. \chi^\nu(\si)\ts,
$$
for all \ts $\la,\mu,\nu\vdash n$.  This only shows that \ts $\{g(\la,\mu,\nu)\}$ \ts
are in \ts $\GapFP$.

Indeed, while the summation above is in \ts $\GapP$ \ts via the MN~rule,
the division by~$n!$ is not allowed in~$\GapP$.  The same issue appears also
when applying \emph{Billey's formula} to compute Schubert coefficients \cite{Bil99},
the \emph{F\'{e}ray--\'{S}niady formula} \cite[Thm~4]{FS11} for the characters,
and Hurwitz's original formula for the \emph{double Hurwitz numbers},
see \cite{GJV05} and further references in \cite[$\S$8.5]{Pak-OPAC}.
All these formulas involve divisions, thus they only prove that the
corresponding integral functions are in \ts $\GapFP$.

We should mention that there are cases when the integrality was established
algebraically.  For example, the integrality of \ts $S_n$ \ts character values:
\ts $\chi^\la(\mu)\in \zz$, follows from the fact that $\si$ and $\si^a$
are conjugate, for all \ts $\si \in S_n$ \ts and \ts $(a,\ord(\si))=1$,
see e.g.\ \cite[$\S$13.1]{Ser77}.  The proof is based on a Galois theoretic
argument combined with a calculation of cyclotomic polynomials, and cannot
be easily translated to a $\GapP$ formula.  In other words, being in $\GapP$
can be a significantly stronger result of independent interest.

Finally, we note that the literature on structure constants discussed above
is much too large to be reviewed in this paper.  We include some additional
references in Section~\ref{s:finrem}.  Our general point stands: we obtain
signed combinatorial interpretations in many cases where none is known, and
new simple signed combinatorial interpretations in a handful of cases where
some are known.

\smallskip

\subsection{Proof ideas} \label{ss:intro-proof}
We start by observing that all structure constants in Theorems~\ref{t:main-quasi}
and~\ref{t:main-schubert} (excluding $\{\Psi_\al,\Phi_\al,\mathfrak{p}_{\al}\}$) are known to be integral.  Furthermore, the proofs that
these are integral involve \emph{unitriangular} \ts changes of bases.
This is best illustrated by the \emph{Kostka matrix} \ts
$(K_{\la\mu})$, which satisfies:
$$
K_{\la\la}\. = \. 1 \quad \text{and} \quad K_{\la\mu}=0 \ \ \text{unless} \ \ \la  \unlhd \mu\ts.
$$

From there, the problem is reduced to making the \defng{M\"obius inversion} \ts effective.
This, in turn, hinges on the fact that partial orders such as ``$\lhd$'' have polynomial
height (cf.~$\S$\ref{ss:finrem-inv}).  In the case of Kostka numbers, this gives an alternative signed combinatorial
interpretation of the inverse Kostka numbers \ts $K_{\la\mu}^{-1}$ \ts that is
different from that in \cite{ER90} and \cite{Duan03}.

Now, for the bases in Theorems~\ref{t:main-quasi} and~\ref{t:main-schubert},
the heavy lifting of establishing the unitriangular property was done
in a series of previous papers, some of them very recent.  Our approach gives
a simple to use tool to extend the unitriangular property to a $\GapP$ result
for the structure constants. This is done in Section~\ref{s:basic}.
In later Sections~\ref{s:sym}--\ref{s:poly}, we derive the results one by one.

Finally, we prove Theorem~\ref{t:main-plethysm} on plethysm coefficients in a short
Section~\ref{s:plethysm}.  The proof uses a simple argument again involving
\ts $\GapP$ \ts formulas for (generalized) Kostka numbers.

\medskip

\section{Basic definitions and notations} \label{s:notation}

We use $\nn=\{0,1,2,\ldots\}$ and $[n]=\{1,\ldots,n\}$.
To simplify the notation,
for a set $X$ and an element $x\in X$, we write \ts $X-x:=X \sm \{x\}$.
Similarly, we write \ts $X+y:=X\cup \{y\}$.

Let \ts $\cP=(X,\prec)$ \ts be a poset on the ground set~$X$ with a partial
order~``$\prec$''. Function \ts $h: X^2\to \rr$ \ts is called
\defn{triangular} \ts w.r.t.~$\cP$ \ts if \ts $h(x,y) = 0$ \ts unless
\ts $x\preccurlyeq y$ \ts for all \ts $x,y\in X$,
and \ts $h(x,x)\ne 0$ \ts for all \ts $x\in X$.
Similarly, function \ts $h: X^2\to \rr$ \ts is called
\defn{unitriangular} \ts w.r.t.~$\cP$ \ts if it is triangular
and \ts $h(x,x)=1$ \ts for all \ts $x\in X$.

Fix~$n$.  An \defn{integer partition} \ts $\la$ of $k$, denoted \ts $\la \vdash k$,
is a sequence of weakly decreasing nonnegative integers \ts
$(\la_1,\ldots,\la_n)$ \ts which sum up to~$k$.  Denote by \ts $U_{n,k}$ \ts the
set of these partitions, and let \ts $\rU_n = \cup_k \ts \rU_{n,k}$\ts.
Similarly a \defn{composition} (sometimes called \emph{weak composition})
\ts $\al$ of $k$, denoted $\al \vDash k$,
is a sequence of nonnegative integers \ts
$(\al_1,\ldots,\al_n)$ \ts which sum up to~$k$.
Let \.  $\vrV_{n,k} \ts :=\{\al\in \nn^n \, : \, \al \vDash k\}$, and \ts let \ts
$\vrV_{n}:=\cup_k \vrV_{n,k}$ \ts be sets of compositions. A \defn{strong composition}
\ts $\al \vDash k$ has all parts strictly positive.
Let \.  $\ComW_{n,k} \ts :=\{\al\in \nn^m_{\geq 1} \, : \, \al \vDash k, \ts  m\leq n\}$, and \ts let \ts
$\ComW_{n}:=\cup_k \ComW_{n,k}$ \ts be sets of strong compositions.
Clearly, every partition of $n$ is also a composition.  Let \ts
$D(\al) : =\{(i,j) \. : \. i\leq \al_j\}$ \ts denote the \defn{diagram} of~$\al$.

We write \ts $|\al| := \al_1+\ldots + \al_n$ \ts the size of the composition,
and \ts $\ell(\al)$ \ts the number of parts in~$\al$.
Denote \ts
$z_\al:=1^{m_1} \ts m_1! \.  2^{m_2} \ts m_2! \. \cdots$, where \ts $m_i=m_i(\al)$ \ts
denotes the multiplicity of~$i$ in~$\al$.
For two compositions \ts $\al,\be\vDash k$, the \defn{dominance order} \ts
is defined as follows:
$$\al\unlhd \be \ \ \Longleftrightarrow \ \
\al_1 \ts + \ts \ldots \ts + \ts \al_i \ts \ge \ts \be_1 \ts + \ts \ldots \ts + \ts \be_i  \ \ \text{for all} \ i.
$$
Note that compositions of different integers are incomparable.

For a permutation \ts $\si\in S_n$\ts, the \defn{Lehmer code} \ts is
a sequence \ts $\bc=(c_1,\ldots,c_n)\in \nn^n$ \ts given by \ts
$c_i(\si) := \#\{j > i : \si_i \ge \si_j\}$.  Note that \ts $|\bc| = \inv(\si)$
\ts is the \defn{number of inversions} in~$\si$.
Denote by \ts $S_{\infty}$ \ts the set of bijections \ts $w:\nn_{\ge 1} \to \nn_{\ge 1}$ \ts
which eventually stabilize: \ts $w(m)=m$ \ts for $m$ large enough.  We view such \ts $w$ \ts
as an infinite word, and note that the number of inversions is well defined on~$S_\infty$\ts.

\defn{Young diagram} \ts of shape \ts $\la$, denoted \ts $[\la]$,
is a set of squares \ts $\{(i,j) \.: \. 1\le i \le \ell(\la), \. 1 \le j \le \la_i\}$.
A \defn{semistandard Young tableau} \ts of shape~$\la$ is a
map \ts $A: [\la] \to \nn$, which is weakly increasing in rows: \ts $A(i,j) \le A(i,j+1)$ \ts
and strictly increasing in columns: \ts $A(i,j) < A(i+1,j)$.  The \defn{content} of~$A$
is a sequence \ts $(m_1,m_2,\ldots)$, where $m_k$ is the number of $k$ in the multiset~$\{A(i,j)\}$.
Let \ts $\SSYT(\la,\mu)$ \ts denote the set of semistandard Young tableaux of shape $\la$
and content~$\mu$.

Consider the polynomial ring $\cc[x_1,\ldots,x_n]$ in $n$ variables.  For an integer sequence
\ts $\ba = (a_1,\ldots,a_n)\in \nn^n$, denote \ts $\bx^\ba:= x_1^{a_1}\cdots x_n^{a_n}$.
We use \ts $[\bx^\ba] \ts F$ \ts to denote the coefficient of \ts $\bx^\ba$ \ts in the
polynomial~$F$.
A polynomial \ts $F\in \cc[x_1,\ldots,x_n]$ \ts is \defn{symmetric} \ts if its coefficients
satisfy
$$
[x_1^{a_1}\cdots x_n^{a_n}] \ts F \. = \.  [x_{\si(1)}^{a_1}\cdots x_{\si(n)}^{a_n}] \ts F \quad
\text{for all} \ \ \si \in S_n \ \ \text{and} \ \ (a_1,\ldots,a_n) \in \nn^n.
$$
The ring of symmetric polynomials is denoted by \ts $\La_n$.   Similarly, $F$ is \defn{quasisymmetric} \ts if
$$
[x_1^{a_1}\cdots x_k^{a_k}] \ts F \. = \.  [x_{i_1}^{a_1}\cdots x_{i_k}^{a_n}] \ts F \quad \text{for all} \ \ i_1 < \ldots < i_k\,,
\ \ 1\le k \le n \ \ \text{and} \ \ (a_1,\ldots,a_k) \in \nn_{\ge 1}^k\..
$$
The ring of quasisymmetric polynomials is denoted by \ts $\QSym_n$.  By definition, we have:
$$\La_n \. \ssu \. \QSym_n \. \ssu \. \cc[x_1,\ldots,x_n].
$$

Class \ts $\FP$ \ts is a class of functions computable in poly-time.
Class \ts $\SP$ \ts is closed under addition and multiplication:
$$
f,\ts g\in \SP \ \ \Longrightarrow \ \ f\ts + \ts g\., \ f\cdot g \in \SP\ts.
$$
Class \ts $\GapP:= \SP\ts - \ts \SP$ \ts is closed under addition, subtraction
and multiplication:
$$
f,\ts g\in \GapP \ \ \Longrightarrow \ \ f\ts \pm \ts g\., \ f\cdot g \in \GapP\ts.
$$
Class \ts $\GapFP$ \ts is a class of rational functions which can be written as \ts
$f/g$ \ts where $f\in \GapP$ \ts and \ts $g\in \FP$.  Clearly,
$$
\SP \ \ssu \ \GapP \ \ssu \ \GapFP\ts.
$$


\medskip

\section{Effective M\"obius inversion} \label{s:basic}

Let \. $X =\cup X_n$, where \. $X_n \subseteq \{0,1\}^n$, be a family of combinatorial
objects. Let \. $\cP:=(X,\prec)$ \.  be a poset such that \ts $x \prec y$ \ts only if
\ts $x,y\in X_n$ \ts for some~$n$.  We use \ts $\cP_n=(X_n,\prec)$ \ts to denote a subposet of~$\cP$.
The \defn{height} \ts of a poset $\cQ$, denoted \ts $\height(\cQ)$, is the size of the maximal chain in~$\cQ_n$\ts.  We say that $\cP$ has \defn{polynomial height},
if \ts $\height(\cP_n)\le C \ts n^c$, for some fixed \ts $C, c>0$.

Let \ts $\de: X^2\to \{0,1\}$ \ts be the \defn{delta function} \ts defined as \ts
$\de(x,y) =1$ \ts if \ts $x=y$, and \ts $\de(x,y) =0$ \ts otherwise.
Let \ts $\xi: X^2\to \{0,1\}$ \ts be the \defn{incidence function} \ts
defined as \ts $\xi(x,y) =1$ \ts if \ts $x\preccurlyeq y$ \ts and \ts
$\xi(x,y) =0$ \ts otherwise.  We say that $\xi$ is \defn{poly-time computable}, if for all \ts
$x,y\in X_n$ \ts the decision problem \ts $[x\preccurlyeq ^? y]$ \ts can be decided in \ts $O(n^c)$ \ts
time, for some fixed \ts $c>0$.

The \defn{M\"obius inverse} \ts is a function \ts $\mu(x,y): X^2 \to \zz$, such that
$$
\sum_{z\in X_n} \. \xi(x,z) \cdot \mu(z,y) \, = \, \de(x,y) \quad
\text{for all} \quad x,y\in X_n\..
$$

\begin{prop} \label{p:mu}
Let \. $\cP:=(X,\prec)$ \.  be a poset with polynomial height, and
suppose that the incidence function~$\xi$ is poly-time computable.
Then the M\"obius inverse function \. $\mu$ \. is in \ts $\GapP$.
\end{prop}

\begin{proof}
For all \. $x,y\in X$, denote by \ts $\cP_{xy}$ \ts the interval in~$\cP$,
and let \ts $h:= \height \big(\cP_{xy}\big)$.
Denote by \ts $\cC_{\ell}(x,y)$ \ts the set of chains \. $x\to z_1 \to \ldots \to z_{\ell-1} \to y$ \.
in \ts $\cP_{xy}$ \ts  of length~$\ell$.
By \emph{Hall's theorem} (see e.g.\ \cite[Prop.~3.8.5]{Sta-EC}), we have:
$$
\mu(x,y) \, = \, \sum_{\ell=0}^h \. (-1)^\ell \. \big|\cC_\ell(x,y)\big|\ts.
$$
We conclude: \. $\mu(x,y) \, = \, \mu_+(x,y) \. - \. \mu_-(x,y)$, where
$$\mu_+(x,y) := \sum_{i=0}^{\lfloor h/2\rfloor} \. \big|\cC_{2i}(x,y)\big|
\quad \text{and} \quad
\mu_-(x,y) := \sum_{i=0}^{\lfloor h/2\rfloor} \. \big|\cC_{2i+1}(x,y)\big|.
$$
Since \ts $h$ \ts is polynomial, we have \ts $\mu_\pm \in \SP$ \ts by definition.
This completes the proof.
\end{proof}

Let \ts $\eta: X^2 \to \zz$ \ts be unitriangular w.r.t.\ $\cP$, i.e.\  \ts $\eta(x,x)=1$ \ts for all \ts
$x\in X$, and \ts $\eta(x,y)\ne 0$ \ts implies \ts $x\preccurlyeq y$, \ts $x,y\in X_n$ \ts
for some~$n$.
The \defn{inverse} \ts of~$\eta$ (in the incidence algebra),
is a function \ts $\rho(x,y): X^2 \to \zz$, such that
$$
\sum_{z\in X_n} \. \eta(x,z) \cdot \rho(z,y) \, = \, \de(x,y) \quad
\text{for all} \quad x,y\in X_n\..
$$

\begin{prop} \label{p:rho}
Let \. $\cP:=(X,\prec)$ \.  be a poset with polynomial height, and
suppose that the incidence function~$\xi$ is poly-time computable.
Suppose function \ts $\eta$ \ts is in \ts $\GapP$.  Then the inverse
function \ts of~$\eta$ \ts is also in \ts $\GapP$.
\end{prop}

\begin{proof}
Denote by \ts $\rho$ \ts the inverse function as in the theorem.
We similarly have:
$$
\rho(x,y) \, = \, \sum_{\ell=0}^h \. \sum_{(x\to z_1 \to \ldots \to z_{\ell-1} \to y) \in \cC_\ell(x,y)} \.
(-1)^\ell \, \eta(x,z_1) \cdot \eta(z_1,z_2) \. \cdots \. \eta(z_{\ell-1},y).
$$
The result follows.
\end{proof}

\medskip

\section{Symmetric polynomials}\label{s:sym}
Let \ts $\La_m=\cc[x_1,\ldots,x_n]^{S_n}$ \ts be the ring of symmetric polynomials.
Denote by \ts $\cq_m$ \ts the poset on partitions \ts $\la \vdash m$ \ts with
dominance order \. $\la \unlhd \mu$, for all \ts $\la,\mu \vdash m$.   It is known
that the dominance order is a lattice, see~$\S$\ref{ss:musings-dom}.
Clearly, \. $\height(\cq_m)=O(m^2)$.  In fact, it was shown in \cite{GK86}, that
\. $\height(\cq_m) = \Theta(m^{3/2})$.

\smallskip

\subsection{Standard bases} \label{ss:sym-standard}
\defn{Monomial symmetric polynomials} \. $\{m_\la\.:\.\la \in \rU_n\}$ \ts are defined as
\[
m_\la(x_1,\ldots,x_n) \. := \. \sum_w \. \bx^{w(\la)},
\]
where the summation is over all \ts $w\in S_n$ \ts giving different
reorderings \ts $w(\la)$ \ts of~$\la$.

\defn{Power sum symmetric polynomials} \. $\{p_\la\.:\.\la \in \rU_n\}$ \.
are defined by
$$
p_{\la}(x_1,\ldots,x_n) \, := \. p_{\la_1}(x_1,\ldots,x_n)\. \cdots \. p_{\la_n}(x_1,\ldots,x_n),
$$
where
\[
p_k(x_1,\ldots,x_n) \, := \, x_1^k \. + \. \ldots \. + \. x_n^k \..
\]
Observe that
$$p_{\la}(x_1,\ldots,x_n) \, = \, \sum_{\mu} \. P(\la,\mu) \. m_{\mu}(x_1,\ldots,x_n),
$$
where \ts $P(\la,\mu)$  \ts is the number of nonnegative integer-valued matrices
with column sum~ $\la$ and row sums~$\mu$,
where each column contains at most one nonzero entry.

Similarly, \defn{elementary symmetric polynomials} \. $\{e_\la\.:\.\la \in \rU_n\}$ \. are
are defined by \.
$$
e_{\la}(x_1,\ldots,x_n) \, := \, e_{\la_1}(x_1,\ldots,x_n) \. \cdots \. e_{\la_n}(x_1,\ldots,x_n),
$$ where
\[
e_k(x_1,x_2,\ldots,x_n) \, := \, \sum_{1\le i_1<\ldots<i_k\le n} \. x_{i_1} \. \cdots \. x_{i_k} \..
\]
Observe that
$$e_{\la}(x_1,\ldots,x_n) \, = \, \sum_{\mu} \. E(\la,\mu) \. m_{\mu}(x_1,\ldots,x_n),
$$
where \ts $E(\la,\mu)$ \ts is the number of $\{0,1\}$-matrices with column sums~$\la$ and row sums~$\mu$.

\defn{Complete homogeneous symmetric polynomials} \. $\{h_\la\.:\. \la \in \rU_n\}$ \. are defined by
$$h_{\la}(x_1,\ldots,x_n) \, := \, h_{\la_1}(x_1,\ldots,x_n)\. \cdots \. h_{\la_n}(x_1,\ldots,x_n),
$$
where
\[
h_k(x_1,x_2,\ldots,x_n) \, := \, \sum_{\mu\vdash k} \. m_\mu\..
\]
Observe that
$$h_{\la}(x_1,\ldots,x_n) \, = \, \sum_{\mu} \. H(\la,\mu) \. m_{\mu}(x_1,\ldots,x_n),
$$
where \ts $H(\la,\mu)$ \ts is the number of nonnegative integer-valued matrices
with column sums~$\la$ and row sums~$\mu$.

Finally, \defn{Schur polynomials} \. $\{s_\la\.:\. \la\in \rU_n\}$ \. can be defined as
\[
s_{\la}(x_1,\ldots,x_n) \, := \, \sum_{\mu} \. K_{\la\mu}\. m_{\mu}(x_1,\ldots,x_n),
\]
where the \defn{Kostka numbers} \ts $K_{\la\mu}$ \ts
compute the number of semistandard Young tableaux of shape~$\la$ and content~$\mu$.
The
\defn{Littlewood--Richardson coefficients} \ts $c^\la_{\mu\nu}$ \ts are
defined by
$$
s_\mu \cdot s_\nu \, = \, \sum_{\la} \. c^\la_{\mu\nu} \. s_\la\..
$$
Recall that \ts $\big\{c^\la_{\mu\nu}\big\}$ \ts are given as the number
of LR-tableaux (we omit the definition), a subset of semistandard
Young tableaux, see e.g.\ \cite{Mac95,Sta-EC}.

\smallskip

\subsection{Structure constants} \label{ss:sym-structure-const}
We start with a traditional approach to structure constants,
which we outline for completeness.

\begin{proof}[Proof of Theorem~\ref{t:main-classic}]
The result for \ts $\{p_\la\}$, \ts $\{e_\la\}$ \ts and \ts $\{h_\la\}$ \ts is trivial.
The definition of \ts $\{m_\la\}$ \ts gives their structure coefficients
\[
T(\la,\mu,\nu) \, := \, \#\big\{(u(\la),w(\mu)) \. : \. u(\la)+w(\mu)=\nu\big\},
\]
implying that they are in~$\SP$.  Finally, the LR-coefficients
\ts $\big\{c^\la_{\mu\nu}\big\}$ \ts are in \ts $\SP$ by the definition
of LR-tableaux.
\end{proof}

We now prove a weaker result, using the effective M\"obius inversion.

\begin{prop}\label{p:Kostka-LR}
The inverse Kostka numbers \ts $\{K^{-1}_{\la\mu}\}$ \ts and the
LR-coefficients \ts $\big\{c^\la_{\mu\nu}\big\}$ \ts are in \ts $\GapP$.
\end{prop}

\begin{proof}
Recall \. $\big(K_{\la\mu}\big)$ \. is unitriangular w.r.t.\ to the dominance order, so
$$
s_\la \, = \, \sum_{\mu \ts \unlhd \ts \la} \. K_{\la\mu} \. m_\mu\..
$$
Thus, the \defn{inverse Kostka numbers} \ts are given by
\begin{equation}\label{eq:inv-Kostka}
m_\la \, = \, \sum_{\mu \ts \unlhd \ts \la} \. K^{-1}_{\la\mu} \. s_\mu\..
\end{equation}
Now Proposition~\ref{p:rho} implies that \ts $\{K^{-1}_{\la\mu}\} \in \GapP$.
We have:
$$
\aligned
s_\mu \cdot s_\nu \, & = \, \Big(\sum_{\tau \unlhd \mu} \. K_{\mu\tau} \. m_\tau\Big)
\cdot  \Big(\sum_{\vk \ts \unlhd \ts  \nu} \. K_{\nu\vk} \. m_\vk\Big) \\
& = \,
\sum_{\tau \unlhd \mu} \. \sum_{\vk \ts \unlhd \ts  \nu} \. \sum_{\si} \.  K_{\mu\tau} \. K_{\nu\vk} \. T(\tau,\vk,\si) \cdot m_\si \\
& = \, \sum_{\tau \ts \unlhd \ts  \mu} \. \sum_{\vk \ts \unlhd \ts  \nu} \. \sum_{\si} \. \sum_{\la \ts \unlhd \ts  \si} \.
K_{\mu\tau} \. K_{\nu\vk} \. T(\tau,\vk,\si) \.  K^{-1}_{\si\la} \cdot s_\la
\endaligned
$$
Since \ts $\{K^{-1}_{\la\mu}\} \in \GapP$ \ts and \ts $\{T(\la,\mu,\nu)\} \in \SP$,
this implies that LR-coefficients are in~$\GapP$.
\end{proof}

\smallskip

\subsection{$(q,t)$ deformations} \label{ss:sym-qt}
Following \cite{Mac88}, \defn{Macdonald symmetric polynomials} \ts
$P_{\la}$ \ts can be defined in terms of semistandard Young tableaux:
\[
P_{\la}(\bx;q,t) \, := \, \sum_{\mu} \. m_{\mu}(\bx) \.
\sum_{T\ts \in \ts \SSYT(\la,\mu)} \. \psi_T(q,t)\ts,
\]
where \ts $\bx = (x_1,\ldots,x_n)$ \ts and
\ts $\psi_T(q,t)$ \ts is a explicit rational function given by
a product formula.  In particular, for fixed \ts $q,t\in \qqq$ \ts such that
\ts $0\le q, t < 1$, this function \ts $\psi: T \to \qqq$ \ts is in \ts $\FP/\FP$.

The \defn{Hall--Littlewood polynomials} \. $P_{\la}(\bx;t)$, see \cite{Lit61},
specialize Macdonald symmetric polynomials by taking \ts $q= 0$.
Similarly, the \defn{Jack symmetric polynomials} \. $P_{\la}(\bx;\al)$,
see \cite{Jac70}, specialize Macdonald symmetric polynomials in another direction:
$$
P_{\la}(\bx;\al) \. := \. \lim_{t\rightarrow 1}P_{\la}(\bx;t^\al,t).
$$

It follows from the definition above and the explicit form of \ts $\psi_T$ \ts
that Macdonald symmetric polynomials are unitriangular in the monomial
symmetric polynomials:
$$
P_{\la}(\bx;q,t)  \. = \. \sum_{\mu\ts \unlhd \ts \la} \. K_{\la\mu}(q,t) \. m_\mu(\bx),
$$
where \ts $K_{\la,\la}(q,t)=1$, see  \cite[Thm~2.3]{Mac95}.
Using the argument in the proof of Proposition~\ref{p:Kostka-LR} gives
the last part of Theorem~\ref{t:main-qt}.
By the specialization to Hall--Littlewood polynomials and Jack polynomials,
we obtain the remaining two parts of the theorem.  We omit the details.

\smallskip

\subsection{Schur $P$-polynomials}\label{ss:finrem-other}
For a partition $\la$ with distinct parts, the
\defn{Schur $P$-polynomial} \ts is given by
$$
P_{\la}(\bx) \, := \, P_{\la}(\bx;-1).
$$
This specialization of the Hall--Littlewood polynomials was defined by Schur
(1911) in the study of projective representation theory of $S_n$\ts.
They are also called \emph{Q-functions}, see \cite[$\S$III.8]{Mac95}.

Note that Schur $P$-polynomials span a subring
of \ts $\La_n\ts$.  The following result follows verbatim the proof
of Theorem~\ref{t:main-qt}.

\begin{cor}\label{cor:Schur-P}
Schur $P$-polynomials have structure constants in \ts $\GapP$.
\end{cor}

\smallskip

\subsection{$(q,t)$ analogues} \label{ss:sym-qt-analogues}
In Theorem~\ref{t:main-qt}, we consider deformations of Schur polynomials,
viewed as bases in~$\La$.  One can also view $(q,t)$ as variables and
extend the results in this direction.  For the Hall--Littlewood polynomials
\ts $P_{\la}(\bx;t)\in \La[t]$, the corresponding
Kostka polynomials \ts $K_{\la,\mu}(t)\in \nn[t]$ \ts are the coefficients
of their expansion in Schur polynomials.  They have a known combinatorial
interpretation by Lascoux and Sch\"utzenberger (see e.g.\ \cite[$\S$III.6]{Mac95}).
Using the LR rule, we conclude:

\begin{prop}\label{prop:q-analogues}
Hall--Littlewood polynomials \ts $\{P_{\la}(t)\}$ \ts
have structure constants in \ts $\SP$.
\end{prop}

Here structure constants form a family of polynomials \ts $\big\{c^\la_{\mu\nu}(t)\in \nn[t]\big\}$.
The proposition states that there is a \ts $\SP$ \ts function \. $f: \{(\la,\mu,k)\} \to \nn$, such that
$$
c^\la_{\mu\nu}(t) \, = \, \sum_{k\in \nn} \. f(\la,\mu,k) \. t^k\ts.
$$

More generally, recall the \emph{modified Macdonald polynomials}
\. $\wt H_{\mu}(\bx;q,t)\in \La[q,t]$,
see e.g.\ \cite[Thm~2.8]{Hag08}.  They are defined so that the corresponding
\defn{$(q,t)$-Kostka polynomials} \.
$\widetilde{K}_{\la\mu}(q,t)\in \nn[q,t]$ \ts become the coefficients
of their expansion in Schur polynomials.  A combinatorial interpretation
for the $(q,t)$-Kostka polynomials
remains open (see e.g.\ \cite[$\S$4.1]{vWil20}).

On the other hand, a \emph{signed}
\ts combinatorial interpretation of \ts $\widetilde{K}_{\la\mu}(q,t)\in \nn[q,t]$ \ts
follows immediately from \emph{Haglund's monomial formula} \ts
\cite[App.~A]{Hag08}, giving a combinatorial interpretation for coefficients
of their expansion in Schur functions, combined with a \ts $\GapP$ \ts formula
for the (usual) inverse Kostka numbers.
Using the LR rule again, we conclude:

\begin{prop}\label{prop:qt-analogues}
Modified Macdonald polynomials \ts $\{\wt H_{\mu}(q,t)\}$ \ts have structure constants in \ts $\GapP$.
\end{prop}

\begin{rem} \label{rem:qt-signs}
It follows from the argument above, finding a combinatorial interpretation
for the $(q,t)$-Kostka numbers would easily imply that a combinatorial
interpretation for the modified Macdonald polynomials.  It would be
interesting to find an unconditional proof of this claim.
\end{rem}

\medskip

\section{Quasisymmetric bases}\label{s:qsym}

\subsection{Posets of interest}\label{ss:qsym-posets}
Recall that \.  $\ComW_{n,k} \ts :=\{\al\in \nn^m_{\geq 1} \, : \, \al \vDash k, \ts  m\leq n\}$, and \ts let \ts
$\ComW_{n}:=\cup_k \ComW_{n,k}$ \ts be sets of strong compositions.
Denote by \ts $\mathcal{Z}_{n,k}=(\ComW_{n,k},\lhd)$ \ts  a poset on strong compositions w.r.t.\ the
dominance order. Let \ts $\cZ_{n}=\cup_k\cZ_{n,k}$.
Clearly, we have \ts $\height(\cZ_{n,k})= O(kn)$.

For $\al\in \nn^m$, define ${\sf{sort}}(\al)$ as the partition formed by listing $\al$ in weakly decreasing order.
For $\al,\be\in \ComW_{n,k}$ we say $\be$ is a \defn{refinement} of $\al$ if one can obtain $\al$ by adding consecutive parts of $\be$. For example, $\be=(1,2,2,1,1)$ refines $\al=(3,3,1)$, but $\be$ does not refine $\ga=(4,1,1,1)$.
This defines the \defn{refinement order} \ts ``$\preccurlyeq$'' \ts on \ts $\ComW_{n,k}$\ts.
Denote by \ts $\cD_{n,k}=(\ComW_{n,k},\unlhd')$ \ts a poset on \ts $\ComW_{n,k}$ \ts where
$$\al\unlhd'\be\., \al,\be\in \ComW_{n,k} \quad \Longleftrightarrow \quad \left\{
\aligned
& {\sf sort}(\be)\lhd {\sf sort}(\al) \ \ \text{if} \ \ {\sf sort}(\be) \ne {\sf sort}(\al), \\
&  \be\unlhd\al \ \ \text{if} \ \ {\sf sort}(\be)= {\sf sort}(\al).
\endaligned
\right.
$$
Observe that \. $\height\big(\cD_{n,k}\big)=O(kn^3)$.

\smallskip

\subsection{Integral quasisymmetric bases} \label{ss:qsym-int-bases}
Let \ts $\QSym_n\subseteq\cc[x_1,\ldots,x_n]$ \ts be the ring of quasisymmetric
polynomials in $n$ variables. The \defn{monomial quasisymmetric polynomials}
\ts $\{M_\al\. : \. \al \in \ComW_{n}\}$ \ts are defined as
\[
M_\al(x_1,\ldots,x_n) \, := \, \sum_{1\ts \le \ts i_1 \ts < \ts \ldots \ts < \ts i_\ell \ts \le \ts n}
x_{i_1}^{\al_1} \. \cdots \. x_{i_\ell}^{\al_\ell}\.,
\]
where  \ts $\ell =\ell(\al) \le n$.
Clearly, \ts $\{M_\al\}$ \ts is a linear basis in \ts $\QSym_n$.

Following \cite{Ges84}, the \defn{fundamental quasisymmetric polynomials} \.
$\{F_{\al} \. : \. \al\in \ComW_{n}\}$ \ts are defined by
\[
F_{\al}(x_1,x_2,\ldots,x_n)\, := \, \sum_{\be\ts\preccurlyeq\ts\al} \. M_{\be}(x_1,x_2,\ldots,x_n).
\]
Following \cite{BBSSZ14}, the \defn{dual immaculate polynomials} \.
$\{\Sc_\al^* \. : \. \al\in \ComW_{n}\}$ \ts can be defined by
\begin{equation*}\label{eq:dualimExp}
    \Sc_\al^*(x_1,x_2,\ldots,x_n) \, := \, \sum_{\be}K_{\al,\be}^I \. M_\be(x_1,x_2,\ldots,x_n),
\end{equation*}
where \ts $K_{\al,\be}^I$ \ts is the number of fillings of the diagram \ts
$D(\al)=\{(i,j) \. : \. i\leq \al_j\}$ \ts with content $\be$
such that entries weakly increase along rows and strictly increase down
the \emph{leftmost} column.

Following \cite{HLMvW11}, the \defn{quasisymmetric Schur polynomials} \.
$\{\mathcal{S}_\al\. : \. \al\in \ComW_{n}\}$ \ts can be defined by
\begin{equation*}\label{eq:qschExp}
    {\mathcal{S}_\al}(x_1,x_2,\ldots,x_n) \, := \, \sum_{\be}K_{\al,\be}^S \. M_\be(x_1,x_2,\ldots,x_n),
\end{equation*}
where \ts $K_{\al,\be}^S$ \ts is the number of fillings $T$ of $D(\al)$ with content $\be$ such that:
\begin{itemize}
    \item entries in $T$ weakly decrease across rows,
     \item entries in $T$ strictly increase down the leftmost column, and
     \item the triple rule holds. That is, embed $T$ in an $\ell(\al)\times\max(\al)$ rectangle, filling each newly added box with $0$. Call this $T'$. Then for $1\leq i<j\leq \ell(\al)$ and $2\leq k\leq \max(\al)$, we have
     \[T'(i,k)\leq T'(j,k)\neq 0 \ \ \Longrightarrow \ \ T'(i,k-1)< T'(j,k).\]
\end{itemize}
Observe that \ts $\{F_{\al}\}$, \ts $\{\Sc_\al^*\}$, and \ts $\{\mathcal{S}_\al\}$ \ts are linear bases in \ts $\QSym_n$.

\begin{ex}
Let \ts $\al=(2,2) \vDash 4$. We have:
    \begin{align*}
        F_{2,2} \, &= \, M_{1, 1, 1, 1} \. + \. M_{1, 1, 2}  \. + \. M_{2, 1, 1} \. + \. M_{2, 2}\ts,\\
        \Sc_{2,2}^*\, &= \, 3\ts M_{1, 1, 1, 1} \. + \.  2\ts M_{1, 1, 2} \. + \.  2\ts M_{1, 2, 1}  \. + \. M_{1, 3}
        \. + \.  M_{2, 1, 1} \. + \.  M_{2, 2}\ts, \text{ and}\\
        \mathcal{S}_{2,2}\, & = \, 2 \ts M_{1, 1, 1, 1} \. + \. M_{1, 1, 2} \. + \. M_{1, 2, 1}
        \. + \. M_{2, 1, 1} \. + \. M_{2, 2}\ts.
    \end{align*}
Let \ts $\be=(1,2,1) \vDash 4$.
Then \ts $K_{\al,\be}^I=2$ \ts as given by the following tableaux:
 \[
\ytableausetup
{boxsize=1em}
{\begin{ytableau}
   1 &  2 \\
   2 &  3
\end{ytableau}}
\qquad
\ytableausetup
{boxsize=1em}
{\begin{ytableau}
    1 &  3 \\
   2 &  2
\end{ytableau}} \raisebox{-0.35cm}{ .}
\]
Similarly, we have \ts $K_{\al,\be}^S=1$ \ts as given by the following tableau:
$$
\ytableausetup
{boxsize=1em}
{\begin{ytableau}
   2 &  2 \\
   3 &  1
\end{ytableau}} \raisebox{-0.35cm}{ .}
$$
\end{ex}

\smallskip
\subsection{Unitriangular property} \label{ss:qsym-int-upper}
Observe that the refinement order is a coarsening of dominance order.
By the definition of fundamental quasisymmetric functions,
we thus have unitriangular property for \ts $\{F_{\al}\}$ \ts w.r.t.\
the dominance order.

\smallskip

 We now prove the corresponding result
for the other two bases of quasisymmetric polynomials.

\smallskip

\begin{prop}\label{prop:quasi1}
Dual immaculate polynomials \ts $\{\Sc_\al^*\}$ \ts
have unitriangular property w.r.t.\ the dominance order:
$$
\Sc_\al^* \ = \, \sum_{\al\ts\unlhd\ts\be} \.  K^I_{\al,\be} \. M_\be \quad \text{and} \quad K^I_{\al,\al}\.= \. 1\ts,
$$
for all \ts $\al \in\ComW_{n,k}$.
\end{prop}

\begin{proof}
By the definition of \ts $\Sc_\al^*\ts$, we have \ts $K_{\al,\be}^I=0$ \ts if \ts $|\be|\neq k$.
It was shown in \cite[Prop.~3.15]{BBSSZ14}, that \ts $K_{\al,\be}^I=0$ \ts
unless $\al$ precedes $\be$ in lexicographic order, and that \ts $K_{\al,\al}^I=1$.
The entries increasing conditions in the definition of \ts $K_{\al,\be}^I$ \ts implies that
if $i$ appears in row~$j$ in a tableau,  then \ts $i\leq j$.
Thus \ts $K_{\al,\be}^I=0$ \ts unless \ts $\al\unlhd\be$.
This completes the proof.
\end{proof}

\smallskip

\begin{prop}\label{prop:quasi3}
Quasisymmetric Schur polynomials \ts $\{\mathcal{S}_\al\}$ \ts have unitriangular
property w.r.t.\ the order \ts $\unlhd':$
    $$(\ast) \qquad
    \mathcal{S}_\al \,
    = \, \sum_{\al\unlhd'\be}\. K_{\al,\be}^S \. M_\be \quad \text{and} \quad K_{\al,\al}^S\.= \. 1\ts,
    $$
for all \ts $\al \in\ComW_{n,k}$.
\end{prop}
\begin{proof}
Suppose a tableau $T$ is counted by \ts $K_{\al,\al}^S$. The second and third
tableau conditions ensure that no entries in~$T$ may repeat within a column.
Since the diagrams are left-aligned, this implies \ts $K_{\al,\be}^S=0$ \ts
unless \ts $\al\unlhd'\be$.  Now, it was shown in \cite[Prop.~6.7]{HLMvW11} that $(\ast)$
holds for the lexicographic order (which strengthens $\unlhd'$),
and that \ts $K_{\al,\al}^S=1$.  This completes the proof.
\end{proof}

\smallskip

\begin{proof}[Proof of Theorem~\ref{t:main-quasi}, first and second part]
The first part is straightforward:
$$M_\al \cdot M_\be \, = \, \sum_{\tau} \. c(\al,\be,\tau) \. M_\tau\.,
$$
where \ts $c(\al,\be,\tau)$ \ts is the number of ways to write
$$(\tau_1,\tau_2,\ldots) \, = \, (\al_1,\al_2,\ldots) \. + \.
(0,\ldots,0,\be_1,0,\ldots,0, \be_2, \ldots)\ts.
$$
For the fundamental quasisymmetric polynomials \ts $\{F_{\al}\}$,
a combinatorial interpretation for the structure constants is
given in \cite[$\S$4]{Ges84}.  In \cite[Cor.~5.12]{AS17}, this
combinatorial interpretation is restated (and reproved).  It
follows from there that the corresponding structure constants
are in~$\SP$.

For the second part, we include details only for the dual immaculate polynomials
\ts $\{\Sc_\al^*\}$.
The result for \ts $\{\mathcal{S}_\al\}$ \ts follows by the same argument,
replacing the dominance order~$\unlhd$ and with $\unlhd '$.

To simplify the notation, write \ts $K_{\al,\be}$ \ts for the
\emph{immaculate Kostka numbers} \ts $K_{\al,\be}^I$.
By Proposition~\ref{prop:quasi1} and their combinatorial interpretation,
\ts $\{K_{\al,\be}\}$ \ts are in~$\SP$.
The inverse coefficients \ts $K_{\al,\be}^{-1}$ \ts are defined by
$$
M_\alpha \, = \, \sum_{\be\unlhd \al } \. K_{\al,\be}^{-1} \. \Sc_\be^*\..
$$
Now Proposition~\ref{p:rho} implies that \. $\{K_{\al,\be}^{-1}\}$ \ts
are in \ts $\GapP$. Denote by  \ts $c^\ga_{\al\be}$ \ts the structure constants
defined by
$$
\Sc_\al^* \cdot \Sc_\be^* \, = \, \sum_{\ga} \.c^\ga_{\al\be} \. \Sc_\ga^* \ts.
$$
This gives:
$$
\aligned
\Sc_\al^* \. \cdot \. \Sc_\be^*  \, & = \, \Big(\sum_{\rho \ts \unlhd\ts \al} \. K_{\al,\rho} \, M_\rho \Big)
\cdot  \Big(\sum_{\omega\ts \unlhd\ts \be} \. K_{\be,\omega} \. M_\omega\Big) \\
& = \,
\sum_{\rho\ts \unlhd\ts \al} \. \sum_{\omega \ts \unlhd\ts \be}  \. \sum_{\tau} \.  K_{\al,\rho} \. K_{\be,\omega} \. c(\rho,\om,\tau) \, M_{\tau} \\
& = \,\sum_{\rho \ts \unlhd\ts \al} \. \sum_{\omega \ts \unlhd\ts \be} \. \sum_{\tau} \. \sum_{\gamma \ts \unlhd\ts  \tau} \.
 K_{\al,\rho} \. K_{\be,\omega} \. c(\rho,\om,\tau) \. K^{-1}_{\tau,\ga} \, \Sc_\ga^*\ts.
\endaligned
$$
Thus, the structure constants \ts $\{c^\ga_{\al\be}\}$ \ts are also in~$\GapP$.
\end{proof}

\begin{rem} \label{rem:quasi-signs}
It is easy to see that structure constants for dual immaculate
and quasisymmetric Schur polynomials can be negative.  Thus,
Theorem~\ref{t:main-quasi} proving their signed combinatorial
interpretation is optimal in this case.
\end{rem}

\smallskip

\subsection{Rational quasisymmetric bases} \label{ss:qsym-rational}
Following~\cite{AWvW23}, \defn{combinatorial quasisymmetric power sums} \ts
$\{\mathfrak{p}_\al\. : \. \al \in \ComW_{n}\}$ \ts are defined as
\[
\mathfrak{p}_\al \,  := \, \sum_{\be} \. K_{\al,\be}^{\mathfrak{p}} \. M_{\be}\ts,
\]
where \ts $K_{\al,\be}^\mathfrak{p}$ \ts is the number of \ts
$\ell(\be)\times\ell(\al)$ \ts matrices \ts $(r_{ij})$ \ts with entries in~$\nn$, such that:
\begin{itemize}
    \item \. $r_{i1}+ \ldots + r_{i \ts \ell(\al)} \ts = \ts\be_i$ \ts for all \ts $1\le i\le \ell(\be)$,
    \item \. ${\sf sort}(\al)_j$ \ts is the only nonzero entry in column $j$, for all \ts $1\le j\le \ell(\al)$, and
    \item \. the word obtained by reading entries top to bottom, left to right is~$\al$.
\end{itemize}
For \ts $\al\preccurlyeq\be$, denote by \ts $\al^{(i)}$ \ts
the corresponding parts in $\al$ which sum to \ts $\be_i$.
For example, for \ts $\al=(1,2,2,1,1)$ \ts and \ts $\be=(3,3,1)$,
we have \ts $\al^{(1)}=(1,2)$, \ts $\al^{(2)}=(2,1)$, and \ts $\al^{(3)}=(1)$.

Following \cite{BDHMN20}, \defn{type 1 quasisymmetric power sums} \.
$\{\Psi_\al\. : \. \al \in \ComW_{n}\}$ \ts are defined as
\begin{align*}
    \Psi_\al \, &:= \, z_\al \. \sum_{\al\preccurlyeq\be} \. \frac{1}{p(\al,\be)} \. M_{\be} \quad \text{ where } \quad p(\al,\be)\, := \, \prod_{i=1}^{\ell(\be)}\Big(\prod_{j=1}^{\ell(\al^{(i)})}\. \sum_{k=1}^{j}\. \al^{(i)}_k\Big).
\end{align*}
Similarly, \defn{type 2 quasisymmetric power sums} \.
$\{\Phi_\al\. : \. \al \in \ComW_{n}\}$ \ts are defined as:
\begin{align*}
    \Phi_\al \, &:= \, z_\al \. \sum_{\al\preccurlyeq\be}\. \frac{1}{s(\al,\be)}\. M_{\be}  \quad \text{ where }
    \quad s(\al,\be) \, : = \, \prod_{i=1}^{\ell(\be)}\Big(\ell(\al^{(i)})!\prod_{j=1}^{\ell(\al^{(i)})}\.\al^{(i)}_j\Big).
\end{align*}
Note that \ts $\mathfrak{p}_\al$ \ts have integer coefficients, while \ts
$\Psi_\al$ \ts and \ts $\Phi_\al$ \ts have rational coefficients.

\begin{ex}
    Let \ts $\al=(1,1,2)$. We have:
    \begin{align*}
        \mathfrak{p}_{1,1,2} \, &= \, 2 \ts M_{1,1,2} \. + \. M_{2, 2},\\
        \Psi_{1,1,2}\, &= \,2\ts M_{1,1,2}  \. + \. M_{2, 2} \. + \. \tfrac{4}{5} \ts M_{1,3}\. + \.  \tfrac{1}{3} \ts M_{4}, \text{ and}\\
        \Phi_{1,1,2}\, &= \,2 \ts M_{1,1,2}  \. + \. M_{2, 2} \. + \. \tfrac{4}{3}\ts M_{1,3}\. + \.  \tfrac{1}{2}\ts M_{4}.
    \end{align*}
For $\be=\al$, we have \ts $K_{\al,\be}^{\mathfrak{p}}=2$, given by the following matrices:
 \[
\begin{pmatrix}
0 & 1 & 0\\
0 & 0 & 1\\
2 & 0 & 0
\end{pmatrix}
\qquad
\begin{pmatrix}
0 & 0 & 1\\
0 & 1 & 0\\
2 & 0 & 0
\end{pmatrix}\raisebox{-0.35cm}{.}
\]
\end{ex}

\smallskip

\begin{prop}\label{prop:quasi2}
The following bases have triangular property w.r.t.\ the dominance order:
         \begin{itemize}
         \item \. combinatorial quasisymmetric power sum \. $\{\mathfrak{p}_\al\}$,
        \item \. type 1 quasisymmetric power sum \. $\{\Psi_\al\}$, and
        \item \. type 2 quasisymmetric power sum \. $\{\Phi_\al\}$.
    \end{itemize}
\end{prop}
\begin{proof}
It is proved in~\cite[Prop.~5.15]{AWvW23}, that \ts $K^{\mathfrak{p}}_{\al,\be}=0$ \.
unless \ts $\al\preccurlyeq\be$.  By definition of \ts $\Psi_\al$ \ts and \ts $\Phi_\al$,
the corresponding Kostka constants \ts $K^{\Psi}_{\al,\be}=K^{\Phi}_{\al,\be}=0$ \ts
unless $\al\preccurlyeq\be$.  Since the refinement order is a coarsening of the dominance order,
this proves the first part of the triangular property.

For the second part, it was shown in \cite[Prop.~5.15]{AWvW23} and \cite{BDHMN20}, that
$$K_{\al,\al}^{\mathfrak{p}} \, = \, \prod_{i=1}^n m_i! \qquad \text{and}
\qquad
K_{\al,\al}^\Psi \, = \, K_{\al,\al}^\Phi \, = \, z_\al \. \prod_{i=1}^{\ell(\al)}\frac{1}{\al_i}\..
$$
This completes the proof.
\end{proof}

\smallskip

\begin{proof}[Proof of Theorem~\ref{t:main-quasi}, third part]
The proof of Theorem~\ref{t:main-quasi} (first and second part) extends similarly to show these structure coefficients are in~$\GapFP$.

However, by \cite[Prop.~3.16]{BDHMN20} and \cite[Eq.~(26)]{BDHMN20}, $\{\Psi_\al\}$ and $\{\Phi_\al\}$ have structure constants in $\SP/\FP$. Similarly \cite[Thm.~5.6]{AWvW23} implies structure constants of $\{\mathfrak{p}_\al\}$ are in $\SP/\FP$.
\end{proof}

\medskip

\section{Polynomial bases}\label{s:poly}

\subsection{Posets of interest}\label{ss:poly-posets}
For \ts $k\in\zz_{>0}$, recall \ts $\vrV_{n,k} \ts :=\{\al\in \nn^n \, : \, \al \vDash k\}$.
Denote by \ts $\mathcal{I}_{n,k}=(\vrV_{n,k},\lhd)$ \ts the poset on weak compositions
with respect to dominance order. Clearly, we have \ts $\height(\mathcal{I}_{n,k})=O(kn)$.
Similarly, let \ts $\vrV_{n}=\bigcup_k \vrV_{n,k}$ \ts and denote by \ts $\mathcal{I}_{n}=(\vrV_{n},\lhd)$ \ts
the poset on weak compositions with respect to the dominance order.
Clearly, we have \ts $\height(\mathcal{I}_{n})=O(n^3)$.

To all \ts $\al\in V_{n,k}$, we can uniquely associate a word \ts $w \in S_{\infty}$ \ts
s.t.\ \ts $\inv(w)=k$, via the inverse map to the Lehmer code: \ts $\al=\code(w)$.
Define the set
\[
L_{n,k}\, := \, \{\al\in \vrV_{n,k}  \ : \ \al_i\leq n-i \ \ \text{for all} \ \ 1\le i \le n\}.
\]
By the definition of the Lehmer code, we have
$$
L_{n,k} =\{\code(w) \ : \ w\in S_n  \ \, \text{s.t.} \ \, \inv(w)=k\}.
$$
Define the \defn{Lehmer poset} \. $\mathcal{L}_{n,k}:=(L_{n,k},\lhd)$.

\smallskip

\subsection{Polynomial bases, first batch}\label{ss:poly-bases-I}
For $\al\in \vrV_{n}$, let ${\rm flat}(\al)$ be the strong composition formed
by removing \ts $0$'s \ts in~$\al$.  Following \cite[$\S$3.4]{Hiv00} (see also \cite{AS17,BJS93}),
the \defn{monomial slide polynomials} \. $\{\mathfrak{M}_{\al} \. : \. \al\in\vrV_{n}\}$ \.
and the \defn{fundamental slide polynomial} \. $\{\FS_{\al} \. : \. \al\in\vrV_{n}\}$ \.
are defined as
\[
  \mathfrak{M}_{\al}(\bx) \, := \, \sum_{\substack{\be\ts \unlhd \ts \al \\ {\rm flat}(\be) \ts = \ts {\rm flat}(\al)}} \bx^\be
  \qquad\text{and} \qquad
\FS_{\al}(\bx) \, := \, \sum_{\substack{\be\ts \unlhd \ts \al \\ {\rm flat}(\be)\ts \preccurlyeq\ts {\rm flat}(\al)}}\bx^\be\..
  \]

We consider certain (semistandard) tableaux \ts
$T: D(\al)\to \nn_{\ge 1}$ \ts of shape \ts $D(\al)$, which we define below.
For a tableau $T$, define the \defn{augmented tableau} $\hat{T}$ to be the tableaux formed by adding a box right before
each row of $D(\al)$, filling the new box in row $i$ with entry~$i$. Following \cite{Mas09}, the \defn{Demazure atoms}
\. $\{{\atom}_{\al}\. : \. \al\in\vrV_{n}\}$ \ts introduced in~\cite{LS90}, can be defined as
\[
{\atom}_{\al}(\bx) \, := \, \sum_{\be}\. K_{\al,\be}^\atom \.\bx^\be.
\]
Here \. $K_{\al,\be}^\atom$ \. is the number of augmented tableaux $T$ of shape $D(\al)$ which satisfy:
\begin{itemize}
    \item \, column entries are distinct
    \item \, entries weakly decrease across rows
    \item \, for \ts all \ts $i<j$ \ts one of the following holds:
    \begin{itemize}
        \item \. $c<b< a$,
        \item \. $a\leq c< b$,
        \item \. $b<a\leq c$,
    \end{itemize}
    where \,  $c\gets T(i,k-1)$, \ts $a\gets T(i,k)$, \ts $b\gets T(j,k)$ \ if \ $\al_i\geq \al_j$\., \\
    and \, \ts $b\gets T(i,k-1)$, \ts $c\gets T(j,k-1)$, \ts $a\gets T(j,k)$ \ if \ $\al_i<\al_j$\..
\end{itemize}

\begin{ex}
For \ts $\al=(0,2,1)$, we have \ts ${\atom}_{\al}(x_1,x_2,x_3) = x_1x_2x_3 + x_2^2x_3$, with monomials
corresponding to the following augmented tableaux:
\[\ytableausetup
{boxsize=1em}
{\begin{ytableau}
   *(lightgray) 1  \\
   *(lightgray) 2 &  2 & 1 \\
   *(lightgray) 3 &  3 \\
\end{ytableau}}
\qquad
{\begin{ytableau}
   *(lightgray) 1  \\
   *(lightgray) 2 &  2 & 2 \\
   *(lightgray) 3 &  3 \\
\end{ytableau}}
\raisebox{-0.7cm}{.}
\]
Here the entries added to each row are drawn in gray.
\end{ex}

\smallskip

\subsection{Polynomial bases, second batch}\label{ss:poly-bases-II}
Let \ts $D\subseteq[n]\times[n]$ \ts be a square diagram where entries (boxes) are
labelled with \ts $\bullet$ \ts or~$\ts\circ$.  We defined two types of moves
on these labeled diagrams as follows.

Take row \ts $i\in[n]$ \ts and box \ts $(i,j)\in D$ that is rightmost with label~$\bullet\ts$.
Let \. $i':=\max\{r\in[i] \, : \, (r,j)\not\in D\}$.
Suppose each \ts $(r,j)\in D$ \ts with \ts $i'+1\leq r \leq i$ has label~$\bullet\ts$.  Define:
\begin{itemize}
    \item the \defn{Kohnert move} \ts on $D$ at $(i,j)$ outputs the diagram \ts $D'=D-(i,j) + (i',j)$.  Let the new box $(i',j)$ have label $\bullet\ts$.
    \item the \defn{K-Kohnert move} \ts on $D$ at $(i,j)$ outputs the diagram \ts $D'=D+ (i',j)$. Let the new box $(i',j)$ have label $\bullet\ts$, and the box $(i,j)$ have reassigned label $\circ\ts$.
\end{itemize}

Let \ts ${\sf Koh}(D)$ \ts denote the set of all diagrams obtainable through applying successive Kohnert moves on diagram~$D$. Similarly,
following~\cite{RY15}, let \ts ${\sf KKoh}(D)$ \ts denote the set of all diagrams obtainable through applying successive Kohnert and K-Kohnert moves on~$D$.

For a subset $S\subseteq[n]\times[n]$, let \ts $\wtx(S)\in \nn^n$ \ts denote the \defn{weight} of~$S$, defined by
\[
\wtx(S)_i\, := \, \# \{(i,j)\in S \ : \ j\in[n]\}.
\]
It was shown in \cite{Koh91}, that the \defn{key polynomials} \. $\{\kappa_\al\. : \. \al\in\vrV_{n}\}$ \. introduced in \cite{LS90}, can be computed combinatorially as follows:
 \[
 \kappa_\al(\bx)  \, := \, \sum_{S\ts\in\ts{\sf Koh}(D(\al))} \. \bx^{\wtx(S)},
 \]
 where \. $D(\alpha):= \big\{(i, j)\in [n]\times [n] \, :  \,  \alpha_i \geq j\big\}$.
Similarly, it was shown in  \cite{RY15,PY23}, that the \defn{Lascoux polynomials} $\{\mathcal{L}_\al\. : \. \al\in\vrV_{n}\}$  \. introduced in \cite{Las00}, can be computed combinatorially as follows:
 \[\mathcal{L}_\al(\bx)  \, := \, \sum_{S\ts\in\ts{\sf KKoh}(D(\al))}(-1)^{|\al|-\#S}\. \bx^{\wtx(S)}.\]

\begin{ex}
For \ts $\al=(0,2,1)$, we have:
\begin{align*}
    \kappa_{(0,2,1)}\. &= \. x_1^2x_2 + x_1x_2^2 + x_1^2x_3 + x_1x_2x_3 + x_2^2x_3, \ \.\text{ and}\\
    \mathcal{L}_{(0,2,1)} \. &= \. \kappa_{(0,2,1)} \. - \. (x_1^2x_2^2 + 2x_1^2x_2x_3 + 2x_1x_2^2x_3) \. + \. x_1^2x_2^2x_3\ts.
\end{align*}
For example, the degree $4$ terms of \ts $\mathcal{L}_{(0,2,1)}$ \ts correspond to the following diagrams:
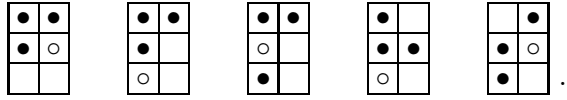
\[\ytableausetup
{boxsize=1em}
{\begin{ytableau}
    \bullet &  \bullet  \\
    \bullet &  \circ  \\
     \  &  \ \\
\end{ytableau}}
\qquad
{\begin{ytableau}
    \bullet &  \bullet  \\
    \bullet &  \ \\
     \circ   &  \ \\
\end{ytableau}}
\qquad
{\begin{ytableau}
    \bullet &  \bullet  \\
    \circ &  \ \\
     \bullet   &  \ \\
\end{ytableau}}
\qquad
{\begin{ytableau}
    \bullet &  \  \\
    \bullet &  \bullet \\
    \circ    &  \ \\
\end{ytableau}}
\qquad
{\begin{ytableau}
    \ &  \bullet  \\
    \bullet &  \circ \\
    \bullet    &  \ \\
\end{ytableau}}\raisebox{-0.7cm}{ .}
\]
\end{ex}

\smallskip

\subsection{Polynomial bases, third batch}\label{ss:poly-bases-III}
Let \ts $D\subset[n]\times[n]$ \ts be a diagram and let \ts $(i,j)\in D$ \ts be a box in the diagram.
The \defn{ladder move} \ts is a transformation \ts $D\ts \to \ts  D-(i,j)+(i-k,j+1)$ \ts and the
\defn{K-ladder move} \ts  is a transformation \ts   $D \ts \to \ts D + (i-k,j+1)$, allowed only
when the following are satisfied:
\begin{itemize}
    \item \ts $(i,j+1)\not\in D$,
    \item \ts $(i-k,j),(i-k,j+1)\not\in D$ \. for some \. $0<k<i$, and
    \item \ts $(i-l,j),(i-l,j+1)\in D$ \. for all \. $0<l<k$.
\end{itemize}

Recall that the \defn{Lehmer code} \. $\code(w) \in \nn^n$ \ts uniquely
determines \ts $w\in S_{\infty}$.
Let ${\sf rPipes}(w)$ denote the set of diagrams obtainable through successive ladder moves, starting from $D(\code(w))$, where $w\in S_n$.
Similarly take ${\sf Pipes}(w)$ to be the set of diagrams obtainable through successive ladder and K-ladder moves, starting from $D(\code(w))$, where $w\in S_n$.

It was proved in \cite{BB93}, that the \defn{Schubert polynomials} \ts $\{\mathfrak{S}_{w} \. : \. w\in S_n\}$ \ts
introduced in~\cite{LS82a}, can be defined as:
\[
\mathfrak{S}_{w}(\bx) \, : = \, \sum_{P\ts \in \ts {\sf rPipes}(w)} \. \bx^{\wtx(P)}.
\]
Similarly, it was proved in \cite{FK94}, that the \defn{Grothendieck polynomials} \ts
$\{\mathfrak{G}_{w} \. : \. w\in S_n\}$ \ts introduced in \cite{LS82b}, can be defined as:
\[
\mathfrak{G}_{w}(\bx) \, = \, \sum_{P\in{\sf Pipes}(w)} \. (-1)^{|\al|\ts - \ts \#P} \. \bx^{\wtx(P)}.
\]
  Since $\code(w)$ uniquely determines $w$, we may write
$\mathfrak{S}_{\al}:=\mathfrak{S}_{\code^{-1}(\al)}$ and $\mathfrak{G}_{\al}:=\mathfrak{G}_{\code^{-1}(\al)}$.

\begin{ex}
Let \ts $w=2143$ \ts and \ts $\al:=\code(w)=(1,0,1,0)$.  We have:
\begin{align*}
    \mathfrak{S}_{\al}(x_1,x_2,x_3) \. &= \. x_1^2 + x_1x_2 + x_1x_3\ts,  \ \. \text{and}\\
    \mathfrak{G}_{\al}(x_1,x_2,x_3) \. &= \. \mathfrak{S}_{\al} \. - \. (x_1^2x_2 + x_1^2x_3 + x_1x_2x_3) \. + \. x_1^2x_2x_3\ts.
\end{align*}
The degree $3$ terms of $\mathfrak{G}_{\al}$ correspond to the following diagrams:
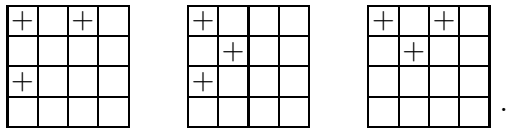
\[\ytableausetup
{boxsize=1em}
{\begin{ytableau}
    + & \ & + & \  \\
    \ & \ & \ & \  \\
    + & \ & \ & \  \\
    \ & \ & \ & \
\end{ytableau}}
\qquad
{\begin{ytableau}
   + & \ & \ & \  \\
    \ & + & \ & \  \\
    + & \ & \ & \  \\
    \ & \ & \ & \
\end{ytableau}}
\qquad
{\begin{ytableau}
  + & \ & + & \  \\
    \ & + & \ & \  \\
    \ & \ & \ & \  \\
    \ & \ & \ & \
\end{ytableau}}\raisebox{-1cm}{ .}
\]
\end{ex}

\smallskip

\subsection{Unitriangular property} \label{ss:poly-unitriangular}
First, we consider homogeneous bases:

\begin{prop}\label{prop:hompoly}
The following linear bases in \ts $\cc[x_1,\ldots,x_n]$ \ts have unitriangular property w.r.t.\ the dominance order $\lhd$ on~$\vrV_{n}:$
         \begin{itemize}
         \item monomial slide polynomials \ts $\{\mathfrak{M}_{\al}\}$,
         \item fundamental slide polynomials \ts $\{\FS_\al\}$,
         \item Demazure atoms \ts $\{{\atomrm}_{\al}\}$,
         \item key polynomials \ts $\{\kappa_\al\}$,
         \item Schubert polynomials \ts $\{\mathfrak{S}_{\al}\}$,
         \item Lascoux polynomials \ts $\{\mathfrak{L}_{\al}\}$, and
    \item Grothendieck polynomials \ts $\{\mathfrak{G}_{\al}\}$.
    \end{itemize}
\end{prop}
\begin{proof}
The results for \ts $\{\mathfrak{M}_{\al}\}$ \ts and \ts $\{\FS_\al\}$ \ts
follow directly from their definition.
For the Demazure atoms polynomials, consider a tableau $T$ counted by $K^\atom_{\al,\be}$.
An entry $i$ in~$T$ must lie in a row weakly below row~$i$.  If \ts $\al\unlhd \be$ fails,
this condition must be broken in every such tableaux, so \ts $K^\atom_{\al,\be}=0$.
Similarly, we have \ts $K^\atom_{\al,\al}=1$, since
the unique valid tableau has all $i$'s in row~$i$.

The result for \ts $\{\kappa_\al\}$ \ts follows from the observation,
that applying a Kohnert move produces a monomial higher in dominance order than~$\al$.
Similarly, for the Schubert polynomials, note that the corresponding
\defn{Schubert--Kostka numbers} \ts $K_{\al,\be}=0$ \ts unless \ts $\al\unlhd\be$.
Indeed, applying a ladder move produces a monomial higher in dominance order than~$\al$.

Finally, the result for \ts $\{\mathfrak{L}_{\al}\}$ \ts follows since applying a
Kohnert  or a K-Kohnert move produces a monomial higher in dominance order than~$\al$.
Similarly, the result for \ts $\{\mathfrak{G}_{\al}\}$ \ts follows since applying
a ladder  or K-ladder move produces a monomial higher in dominance order than~$\al$.
\end{proof}

\smallskip

The proof below follows the idea of the proof of Proposition~\ref{p:Kostka-LR}.

\begin{proof}[Proof of Theorem~\ref{t:main-schubert}]
The first results
follow from Theorems~5.5 and~5.11 in~\cite{AS17}.
For the remaining results, we include details for the Schubert polynomials \ts $\{\mathfrak{S}_{\al}\}$.
The result for other bases follows by the same argument.
By Proposition~\ref{prop:hompoly}, we have:
$$
\Sc_\al(\bx) \, = \, \sum_{\al\unlhd\omega} \. K_{\al,\omega} \. \bx^\omega.
$$
By the \emph{pipe dream} combinatorial interpretation above,
the \defn{Schubert--Kostka numbers} \. $\{K_{\al,\omega}\}$ \ts are in~$\SP$.
The \defn{inverse Schubert--Kostka numbers} \ts are defined by
$$
\bx^\al \, = \, \sum_{\al \ts\unlhd\ts \omega } \. K^{-1}_{\al,\omega} \. \Sc_{\omega}.
$$
Again Proposition~\ref{p:rho} implies that \. $K^{-1} \in \GapP$.

Recall that the \defn{Schubert coefficients} \ts $\{c^\ga_{\al\be}\}$ \ts are
defined by
$$
\Sc_\al \cdot \Sc_\be \, = \, \sum_{\ga} \.c^\ga_{\al\be} \. \Sc_\ga\..
$$
We have:
$$
\aligned
\Sc_\al \cdot \Sc_\be  \, & = \, \Big(\sum_{\al\ts \unlhd\ts\omega} \. K_{\al,\omega} \. \bx^\omega\Big)
\cdot  \Big(\sum_{\be\ts \unlhd\ts\rho} \. K_{\be,\rho} \. \bx^\rho\Big) \,  = \,
\sum_{\al\ts \unlhd\ts\omega} \. \sum_{\be\unlhd\rho} \.  K_{\al,\omega} \. K_{\be,\rho} \. \bx^{\omega+\rho} \\
& = \,\sum_{\al\ts \unlhd\ts\omega} \. \sum_{\be\ts \unlhd\ts\rho}  \. \sum_{\omega+\rho \ts \unlhd\ts \ga} \.
 K_{\al,\omega} \. K_{\be,\rho} \.  K^{-1}_{\omega+\rho,\ga} \cdot \Sc_\ga\ts.
\endaligned
$$
Thus Schubert coefficients \ts $\{c^\ga_{\al\be}\}$ \ts are in~$\GapP$, as desired.
\end{proof}

\begin{rem}
The poset of monomials is also unitriangular w.r.t.\ the reverse lexicographic order, so
the M\"obius inversion can also be used in this setting.  However, the height
of the resulting poset is exponential, so Proposition~\ref{p:rho} is not applicable.
For example, for \ts $w\in S_{2n}$ \ts such that \ts $\code(w)=(0^{n-1},n,0^{n})$,
the monomial support of \ts $\Sc_w$ \ts corresponds to all weak compositions of~$n$,
the number of which is exponential in~$n$.
\end{rem}

\begin{rem}\label{rem:inv-Kostka}
In some cases, the (generalized) inverse Kostka numbers is computed explicitly.
Notably, in \cite[$\S$5]{NT23}, the authors use the lattice structure and the
Crosscut Theorem (see~$\S$\ref{ss:finrem-inv}) to prove that for the fundamental
slide polynomials, all inverse Kostka numbers are in $\{0,\pm 1\}$,
cf.~$\S$\ref{ss:musings-dom}.  See also
inverse Kostka numbers for \emph{forest polynomials} given in \cite[Prop.~10.16]{NST24}.
\end{rem}

\begin{rem}\label{rem:Schubert-signs}
It is easy to see that the structure constants for Demazure atoms,
key, and Lascoux polynomials can be negative without predictable signs.
Thus,  Theorem~\ref{t:main-quasi} proving their signed combinatorial
interpretation is optimal in this case.  For Schubert polynomials
the structure constants are always positive, and
conjectured to be not in~$\SP$ by the first author
\cite[Conj.~10.1]{Pak-OPAC} (cf.\ \cite[Problem~11]{Sta00}).
For Grothendieck polynomials the structure constants have predictable signs
(see e.g.~\cite{Bri02}).  Adjusting for signs, whether these have
a combinatorial interpretation also remains a major open problem.
\end{rem}

\medskip

\section{Plethysm} \label{s:plethysm}

The following proof streamlines and extends the argument in \cite[$\S$9]{FI20}.

\begin{proof}[Proof of Theorem~\ref{t:main-plethysm}]
    Since $\{f_\la\}$, $\{g_{\mu}\}$ are linear bases in~$\La_n$\ts, we can write
    \[
    f_\la[g_{\mu}] \, = \, \sum_{\nu} \. d_{\la\mu}^\nu  \. m_{\nu}\ts.
    \]
By \eqref{eq:inv-Kostka},
the result follows once we show that coefficients \ts $\{d_{\la\mu}^\nu\}$ \ts
are in $\GapP$.

Using combinatorial interpretations in Section~\ref{s:sym},
let \ts $Y^{(g)}_\mu = \{\tau,\tau',\tau'',\ldots\}$ \ts
be a set of monomials in $g_\mu:$
$$
g_{\mu} \, = \, m_\tau \. + \.  m_{\tau'} \. + \.  m_{\tau''} \. + \. \ldots
$$
Denote \ts $K_{\la,\rho}^{(f)}:=|Y^{(f)}_\la| = [\bx^{\rho}]\ts f_{\la}$.  Then:
    \begin{align*}
        f_\la[g_{\mu}]\, &= \, f_\la\big(\bx^{\tau},\bx^{\tau'},\ldots\big)
        \, = \, \sum_{\rho } \. K_{\la,\rho}^{(f)} \. m_{\rho}\big(\bx^{\tau},\bx^{\tau'},\ldots\big) \\
        &= \, \sum_{\rho } \. K_{\la,\rho}^{(f)} \. \sum_{w} \. \bx^{w(\rho)_1 \. \tau} \cdot \bx^{w(\rho)_2 \. \tau'} \. \cdots
        \end{align*}
Taking coefficients in \. $\bx^\nu$  on both sides implies that \. $\{d_{\la\mu}^\nu\} \in \SP$.  This completes the proof.
\end{proof}

\begin{rem}\label{rem:plethysm-signs}
It remains a major open problem whether there is a combinatorial interpretation
for plethysm coefficients, even in a special case when \ts $\la=(r^{n/r})$ \ts
is a rectangle, see \cite[Problem~9]{Sta00}.
The first author conjectured that these are not in~$\ts \SP$ \ts
\cite[Conj~8.8]{Pak-OPAC}.
\end{rem}

\medskip

\section{Further applications} \label{s:app}

We conclude with one additional application relating two type of structure
constants: into symmetric and quasisymmetric polynomials.

\begin{thm}\label{t:main-symQsymCOB}
Let \ts $A=\cup A_n$ \ts be a family of combinatorial objects, and let
\ts $\{G_w(\bx) \. : \. w\in A\}$ \ts be a family of symmetric polynomials
such that
$$
G_w(\bx) \. = \. \sum_{\al\ts\in \ts \vrV_n} \. c_{w\al} \ts F_{\al}(\bx)\ts,
$$
where the coefficients \ts $\{c_{w\al}\}$ \ts are in $\GapP$.
Consider the coefficients defined by
$$
G_w(\bx) \. = \. \sum_{\la\ts\in \ts \rU_n} \. d_{w\la} \ts s_{\la}(\bx)\ts.
$$
Then \ts $\{d_{w\la}\}$ \ts are also in $\GapP$.  Furthermore, the result holds
when \ts $\{F_{\al}\}$ \ts are replaced with \ts $\{M_\al\}$.
\end{thm}

For example, this theorem gives another proof of Proposition~\ref{prop:qt-analogues}.

\begin{proof}[Proof of Theorem~\ref{t:main-symQsymCOB}]
Take  \ts $\la\vdash m$.
By \cite[Thm~11]{ELW10}, we have:
$$
d_{w\la} \, = \,
    \sum_{\al\ts \ts \vDash m}\. \sum_{\be\preccurlyeq \al} \. K(\be,\la) \ts c_{w\al}\,,
$$
where  \ts $\{K(\be,\la)\}$ \ts are in $\GapP$. In fact, \ts $K(\be,\la)$ \ts
have an explicit signed combinatorial interpretation as a signed sum of
certain rim-hook tableaux.  This proves the first part.

For the second part, consider the  inverse coefficients \ts $K_{\al,\be}^{-1}$ \ts defined by
$$
M_\alpha(\bx) \, = \, \sum_{\be\unlhd \al } \. K_{\al,\be}^{-1} \. F_\be(\bx).
$$
By Proposition~\ref{p:rho}, coefficients \ts $\{K_{\al,\be}^{-1}\}$ \ts
are in \ts $\GapP$.  The result now follows from the first part.
\end{proof}

\begin{rem}\label{rem:LLT}
Combining Theorem~\ref{t:main-symQsymCOB} with \cite[Eq.~(82)]{HHL05} gives a \ts $\GapP$ \ts
formula for the Schur expansion of \emph{LLT polynomials}. While this expansion has been
proven to be Schur-positive \cite{GH06}, there is no known (unsigned) combinatorial interpretation
for this expansion (see e.g.\ \cite[$\S$7]{AU20}).
\end{rem}

\medskip

\section{Final remarks and open problems}\label{s:finrem}

\subsection{Historical notes}  \label{ss:finrem-hist}
Many different generalizations of Kostka numbers (weight multiplicities)
and the LR--coefficients (structure constants) have been studied
across the area, too many to review here.  In connection to
our Theorem~\ref{t:main-qt}, the most notable (and most general) are
generalizations of the LR rule to Hall--Littlewood polynomials \cite{Sch06},
to Macdonald polynomials \cite{Yip12},
and to Koornwinder polynomials \cite{Yam22} (see also references therein).

A generalization of the \emph{inverse Kostka numbers} \ts to the
Hall--Littlewood polynomials is given in \cite{Car98},
a quasisymmetric version is given in~\cite{ELW10},
and a noncommutative version is given in~\cite{AM23}.  A curious
signed combinatorial interpretation for the Kerov character polynomials
was given in \cite{Fer15}.

For Schubert polynomials and their structure constants, the literature
is again much too large to review.  We refer to \cite{Knu22} for an overview,
and to \cite{KZ23} for a recent breakthrough which includes a new type
of signed combinatorial interpretation in a special case.

For the plethysm, the literature is again much too large to review.
We refer to \cite{CGR84,LR11,Yang98} for algorithms
implying that several plethysm coefficients are in \ts $\GapP$.
We warn the reader that a formula in \cite[Thm~4.1]{Dor97}
does not give a \ts $\GapP$ \ts formula, at least not without
effort.  We refer to \cite{C+22} for a recent overview of the area.

\subsection{Dominance order}  \label{ss:musings-dom}
It was shown by Bogart (unpublished) and Brylawski \cite{Bry73}, that
the M\"obius function for the dominance order \ts $\cQ_n$ satisfies
$$
\mu(\al,\be) \. \in \. \{0,\pm 1\} \ \ \text{for all} \ \ \al, \ts \be \vdash n.
$$
This was reproved and generalized in a series of papers, including
\cite{BS97,Gre88,Kahn87,Ste98}.  The proofs also give easy \ts poly$(n)$ \ts
time algorithms for computing \ts $\mu(\al,\be)$.  See also
Remark~\ref{rem:inv-Kostka} for connections to the inverse
Kostka numbers.

There seem to be no closed formula for
the  M\"obius function on \ts $\cZ_{n,k}$ \ts
(cf.\ \cite{SV06} for a different order on strong compositions).
It would be interesting to see if this M\"obius function is
bounded.
We note that \. $\pp[\la \lhd \mu]\to 0$ \. for uniform \.
$\la,\mu \vdash n$\ts, as \ts $n \to \infty$,  \. \cite{Pit99}.  Only recently,
it was shown that \ts $\pp[\la \lhd \mu]= n^{\Theta(1)}$ \ts \cite{MMM21},
but the lower and upper bounds remain far apart.
It would also be interesting to find the corresponding result
for  \ts $\cZ_{n,k}\ts$.

\subsection{M\"obius inversion for lattices}  \label{ss:finrem-inv}
As we mentioned above, it is well known that the dominance order
\ts $\cq_m$ \ts is a lattice.  This holds for other posets
we consider in this paper.  The proof of the following result is
straightforward and will be omitted.

\begin{prop}\label{p:lattices}
Partial orders \. $\cZ_{n,k}$, \. $\mathcal{I}_{n,k}$,
\. and \. $\mathcal{L}_{n,k}$ \. are lattices.
\end{prop}

For lattices of polynomial width, an alternative approach to the effective
M\"obius inversion is given by the \emph{Crosscut Theorem}
\cite[Cor.~3.9.4]{Sta-EC}.  Since the lattices in Proposition~\ref{p:lattices}
have exponential width, this approach does not give a new family
of \ts $\GapP$ \ts formulas for the structure constants.

\subsection{Signed combinatorial interpretations}\label{ss:finrem-signed}
Although our $\GapP$ formulas tend to be rather complicated, in most
cases they give the first signed  combinatorial interpretations for these
structure constants.  In fact, one can simplify one part of it as follows.
Suppose $f=g-h$ is a $\GapP$ formula, $g,h\in \SP$,  so that
\ts $h(w) = \#\{u \. : \. (w,u) \in B\}$ \ts for some \ts $\NP$ \ts language \ts $B\subseteq W^2$.
Let \ts $a,C\in \nn_{\ge 1}$  \ts be s.t.\ \ts $|u|\le C\ts |w|^a$ \ts
in the definition of \ts $h(w)$.  Observe that \. $2^{C n^a}-h = \#\{u \.:\. (w,u) \notin B\} \in \SP$,
for all \ts $|w|=n$.  We conclude that \ts $f=\big[g+(2^{C n^a}-h)\big] - 2^{C n^a}$ \ts
is a $\GapP$ formula with the negative part in~$\FP$.  In other words, adding a
sufficiently large power of two can turn a signed combinatorial interpretation into
the usual combinatorial interpretation.

\subsection{Binary encoding}\label{ss:finrem-binary}
Both the Kostka and the LR-coefficients are in $\SP$ in binary when written
as Gelfand--Tseitlin patterns.  It follows from the signed combinatorial
interpretation of the inverse Kostka numbers given in
\cite{ER90}, that  \. $K^{-1} \in \GapP$ \. in binary.  This can also be
derived from an argument in the proof of Proposition~\ref{p:Kostka-LR},
but an extra effort is needed. The same applies to other structure
constants in the paper.

\subsection{Quantum complexity}\label{ss:finrem-quantum}
By analogy with $\GapP$, one can similarly ask whether the structure
constants considered in the paper are in the quantum complexity classes \ts
$\SBQP$ \ts and \ts $\SBQP/\FP$.  This was proved for Kronecker and
plethysm coefficients in \cite{IS23} (see also \cite{B+23}).

\subsection{Combinatorial interpretations}\label{ss:finrem-CI}
Finding the right place for a combinatorial function is the first step towards
understanding its true nature.  Theorems~\ref{t:main-classic}--\ref{t:main-plethysm}
give the best inclusions we know for combinatorial interpretations of
the structure constants (see Remarks~\ref{rem:qt-signs},
\ref{rem:quasi-signs}, \ref{rem:Schubert-signs} and~\ref{rem:plethysm-signs}).
Restricting to integral functions, from the computational complexity point of view
our examples of fall into four buckets.

First, there are $\GapP$ functions where the sign is
computationally hard to determine.  Second, there are $\SP$ functions
where the sign is easy to compute and the absolute value
is in~$\SP$.  Third, there are nonnegative \ts $\GapP$ \ts
functions which have no \ts $\SP$ \ts formula (modulo standard
complexity assumptions).  Finally, there are nonnegative \ts
$\GapP$ \ts functions which have no \emph{known} \ts $\SP$ \ts formula.

\vskip.5cm

\subsection*{Acknowledgements}
We are grateful to Jim Haglund, Christian Ikenmeyer, Nick Loehr
and Alex Yong for interesting discussions and helpful remarks.
Special thanks to Darij Grinberg and Vasu Tewari for
careful reading of the paper and useful comments, and to Greta Panova
for suggesting Theorem~\ref{t:main-symQsymCOB}.
We benefitted from the symmetric functions catalogue \cite{Ale20}. We thank the anonymous referee for their helpful remarks. 

The first author was partially supported by the NSF grant CCF-2007891.
The second author was partially supported by the NSF MSPRF grant DMS-2302279.
This paper was finished while both authors attended the IPAM semester
in \emph{Geometry, Statistical Mechanics, and Integrability}.  We are grateful
for the hospitality.

\end{document}